\documentclass[11pt]{article}
\usepackage[margin=3cm]{geometry}
\usepackage{booktabs}
\usepackage{longtable}
\usepackage[T1]{fontenc}
\usepackage{lmodern}
\usepackage[numbers,sort&compress]{natbib}
\usepackage{algorithm,setspace}
\usepackage{algpseudocode}
\usepackage{amsmath, amssymb, amsthm}
\usepackage{caption,subcaption}
\usepackage{mathrsfs}
\usepackage{stmaryrd}
\usepackage{color, graphicx, float}
\usepackage[breakable]{tcolorbox}
\usepackage{tikz}
\usetikzlibrary{positioning, quotes}

\newlength{\R}

\usepackage{etoolbox}
\makeatletter
\patchcmd{\@maketitle}{\LARGE \@title}{\LARGE\bfseries\@title}{}{}

\renewcommand{\@seccntformat}[1]{\csname the#1\endcsname.\quad}
\makeatother

\usepackage{hyperref}
\hypersetup{
    colorlinks=true,
    linkcolor=blue,   
    citecolor=red,      
    urlcolor=blue        
}

\makeatletter
\def\th@plain{%
	\thm@notefont{}
	\itshape 
}
\def\th@definition{%
	\thm@notefont{}
	\normalfont 
}

\renewenvironment{proof}[1][\proofname]{\par
	\normalfont
	\topsep0\p@\@plus3\p@ \trivlist
	\item[\hskip\labelsep\itshape
	#1\@addpunct{.}]\ignorespaces
}{%
	\qed\endtrivlist
}
\makeatother

\newtheorem{theorem}{Theorem}[section]
\newtheorem{lemma}[theorem]{Lemma}

\theoremstyle{definition}

\theoremstyle{definition}
\newtheorem{example}[theorem]{Example}
\theoremstyle{definition}
\newtheorem{remark}[theorem]{Remark}
\theoremstyle{definition}

\usepackage[shortlabels]{enumitem}

\allowdisplaybreaks

\makeatletter
\skip\footins=8pt plus 2pt minus 2pt        
\makeatother

\begin{document}
\title{On Generalized Forward-Reflected-Backward Method for Monotone Inclusion Problems}

\author{
Santanu Soe$^{1,2}$\!,
\;V.~Vetrivel$^{1}$\!,
\;Jen-Chih Yao$^{3}$
}
\date{\today}

\makeatletter
\renewcommand{\thefootnote}{\arabic{footnote}}
\makeatother

\maketitle

\begingroup
\setlength{\parskip}{0pt} 

\footnotetext[1]{\raggedright
\textit{Department of Mathematics, Indian Institute of Technology Madras, Chennai 600036, Tamil Nadu, India.}
Email (Santanu Soe): \href{mailto:ma22d002@smail.iitm.ac.in}{ma22d002@smail.iitm.ac.in};\;
Email (V.~Vetrivel): \href{mailto:vetri@iitm.ac.in}{vetri@iitm.ac.in}.}

\footnotetext[2]{\raggedright
\textit{School of Mathematics and Statistics, The University of Melbourne, Parkville, VIC 3010, Australia.}
Email (Santanu Soe): \href{mailto:santanu.soe@student.unimelb.edu.au}{santanu.soe@student.unimelb.edu.au}.}

\footnotetext[3]{\raggedright
\textit{Center for General Education, China Medical University, Taichung 40402, Taiwan.}
Email (Jen-Chih Yao): \href{mailto:yaojc@mail.cmu.edu.tw}{yaojc@mail.cmu.edu.tw}.}
\endgroup

\setcounter{footnote}{0}
\renewcommand{\thefootnote}{\arabic{footnote}}

\maketitle

\begin{abstract}
We study the generalized forward–reflected–backward (GFRB) method, an extension of the forward–reflected–backward (FRB) scheme due to Malitsky and Tam, for solving monotone inclusion problems in real Hilbert spaces. We first analyze GFRB equipped with a non-decreasing step-size rule that does not require prior knowledge of the Lipschitz constant of the operator involved. We then present two illustrative examples: in the first, we show that the convergence rate of GFRB is bounded from below by that of FRB, and in the second, we obtain an improved convergence rate for GFRB via an appropriate choice of initial parameters. In the sequel, we propose an extended primal–dual twice-reflected (PDTR) algorithm and show that it can be recovered from GFRB under suitable metric selections. Finally, we validate the proposed approach on several state-of-the-art problems and demonstrate better numerical performance compared to the existing ones. 
\end{abstract}



\section{Introduction}
Let $\mathbb{H}$ be a real Hilbert space. Let us consider the following monotone inclusion problem of finding $x\in\mathbb{H}$ such that
\begin{equation}\label{main}
    0\in (A+B)(x),
\end{equation}
where $A: \mathbb{H}\rightrightarrows\mathbb{H}$ is a maximal monotone operator and $B: \mathbb{H} \to\mathbb{H}$ is a monotone, $L$-Lipschitz continuous operator. Throughout this paper, we assume that the solution set of \eqref{main} is nonempty, i.e., $(A+B)^{-1}(0)\neq\varnothing$. The inclusion problem \eqref{main} encompasses numerous classes of problems, such as monotone variational inequalities, saddle point problems, equilibrium problems, machine learning problems, and so on; see, e.g., \cite{malitsky2023first,bauschke2011convex,briceno2011monotone+,briceno2018forward,bui2022multivariate,malitsky2015projected,malitsky2020forward}. One of the best-known methods for solving~\eqref{main}, under the assumption that $A$ is maximally monotone and $B$ is monotone and $\beta$-cocoercive, is the \emph{forward--backward} (FB) method proposed by Lions and Mercier~\cite{lions1979splitting}. Given $x_0\in\mathbb{H}$, the iteration scheme of FB is given by 
\begin{equation}\label{FB}
     x_{k+1}= J_{\lambda A}(x_k-\lambda B(x_k)),
\end{equation}
where $\lambda>0$ and $J_{\lambda A} := (I+\lambda A)^{-1}$ denotes the resolvent of $A$. Under the cocoercivity assumption on $B$, the sequence generated by~\eqref{FB} converges weakly to a solution of~\eqref{main}, provided that the step-size $\lambda$ satisfies $0 < \lambda < \frac{2}{\beta}$. Since cocoercivity is a stronger notion than Lipschitzness, if one merely assumes that $B$ is monotone and Lipschitz continuous, the forward-backward method might fail to converge, see \cite[Example 12.1.3]{facchinei2003finite} for a counterexample. In order to resolve this issue, Tseng \cite{tseng2000modified} proposed a variant of FB, known as the \emph{forward-backward-forward} (FBF) method.  Given $x_0\in\mathbb{H}$, the FBF method iterates as
\begin{align}\label{FBF}
    \begin{cases}
        y_k = J_{\lambda A}(x_k-\lambda B(x_k)),\\
        x_{k+1}= y_k-\lambda B(y_k) + \lambda B(x_k).
    \end{cases}
\end{align}
Under the assumption that $A$ is maximally monotone and $B$ is monotone and $L$-Lipschitz, Tseng's FBF method \eqref{FBF} converges weakly to a solution of \eqref{main} provided the step-size $\lambda$ satisfies $0<\lambda<\frac{1}{L}$. Although Tseng's FBF method comes with the advantage of relaxation on the assumptions on $B$, it requires two forward evaluations of $B$, which can be computationally expensive when $B$ has a complex structure. This limitation was addressed by Malitsky and Tam \cite{malitsky2020forward}, who introduced the
\emph{forward--reflected--backward} (FRB) method. Starting from initial points
$x_0,x_{-1}\in\mathbb{H}$, the FRB iteration reads
\begin{equation}\label{FRB}
x_{k+1} = J_{\lambda A}\left(x_k-2\lambda B(x_k)+2\lambda B(x_{k-1})\right),
\end{equation}
where $\lambda>0$ is a given step-size. Under the same assumptions on $A$ and $B$ as in Tseng's FBF \eqref{FBF}, the sequence generated by \eqref{FRB} converges weakly to a solution of \eqref{main}, provided the step-size $\lambda$ satisfies $0<\lambda<\frac{1}{2L}$. Recently, the FRB method~\eqref{FRB} was generalized by stressing the reflection term $B(x_k) - B(x_{k-1})$ and incorporating an additional reflection term from a previous iterate, 
$B(x_{k-1}) - B(x_{k-2})$. 
This variant, called the \emph{generalized forward--reflected--backward} (GFRB)~\cite{dung2025generalization} algorithm, preserves all the properties of FRB \eqref{FRB} and has the same computational cost; namely, one evaluation of $B$ and one resolvent computation of $A$ per iteration. Given initial points $x_0, x_{-1}, x_{-2} \in \mathbb{H}$, $\delta\in\mathbb{R}$ and $\alpha\in[0,1)$, the iteration of GFRB is as follows
\begin{equation}\label{GFRB}
    x_{k+1} = J_{\lambda A} \left( (1-\alpha)x_k +\alpha x_{k-1} - \lambda(\delta+2) B(x_k) + \lambda(2\delta+1) B(x_{k-1}) - \lambda \delta B(x_{k-2}) \right).
\end{equation}
Under the same assumptions on $A$ and $B$ as in FRB and Tseng's FBF method, the sequence generated by \eqref{GFRB} converges weakly to a solution of \eqref{main} provided the step-size $\lambda$ satisfies $0<\lambda<\frac{1-\alpha}{2L(1+|\delta|)}$. Note that plugging $\delta=\alpha=0$ in \eqref{GFRB} simplifies both the algorithm and the corresponding step-size condition to FRB \eqref{FRB}, which suggests the name of the algorithm.

There are several other methods that can efficiently solve \eqref{main}. Among them, recently, Cevher and Vu \cite{cevher2021reflected} proposed the \emph{reflected forward–backward} (RFB) method by extending the reflected projected gradient method of Malitsky \cite{malitsky2015projected} for variational inequalities to solve the monotone inclusion problem \eqref{main}. Given initial points $x_{-1},x_0\in\mathbb{H}$, RFB iterates as
\begin{align}\label{RFB}
    \begin{cases}
        y_k = 2x_k - x_{k-1},\\
        x_{k+1} = J_{\lambda A}\bigl(x_k - \lambda B(y_k)\bigr),
    \end{cases}
\end{align}
 Under the same assumptions on $A$ and $B$ as for FRB \eqref{FRB} and GFRB \eqref{GFRB}, the sequence $(x_k)$ generated by \eqref{RFB} converges weakly to a solution of \eqref{main}, provided the step-size $\lambda$ satisfies $0<\lambda<\frac{\sqrt{2}-1}{L}$. Several modified versions of RFB~\eqref{RFB} for solving \eqref{main} have also been proposed; see, e.g., \cite{izuchukwu2025application,shehu2024outer}.

Nonetheless, a recurring limitation of many forward--backward type schemes is that their implementable stepsize rules often rely on \emph{a priori} knowledge of the Lipschitz constant $L$ of the single-valued operator $B$. To overcome this requirement, several researchers have proposed techniques that eliminate the need for such prior knowledge. For instance, Tan et al.~\cite{tan2025perturbed} proposed a perturbed reflected forward--backward variant for \eqref{main} that avoids using $L$ explicitly. Malitsky and Tam~\cite{malitsky2020forward} introduced a linesearch for the FRB method~\eqref{FRB}, thereby removing the dependence on a the Lipschitz bound. More recently, Lu and Mei~\cite{lu2025primal} developed primal--dual extrapolation splitting methods for the monotone inclusion~\eqref{main} and derived corresponding iteration-complexity guarantees. While linesearch procedures are robust, they typically introduce practical overhead: one must choose initial trial parameters, evaluate operators multiple times per outer iteration, and run an inner loop until a termination condition is met, which often increases the computational cost.

Motivated by these concerns, Hieu et al.~\cite{hieu2021modified} proposed a \emph{nearly decreasing} stepsize rule for FRB~\eqref{FRB} that dispenses with prior knowledge of $L$. However, decreasing stepsizes may lead to slower practical performance, especially when the initial stepsize is chosen conservatively; subsequent iterates remain restricted by this small bound and may require many iterations to reach a prescribed accuracy. One way to mitigate this effect is to incorporate inertial terms, as in the momentum-based FRB variants studied by Yao et al.~\cite{yao2024forward}. Along a different direction, Hoai~\cite{hoai2025new} proposed a \emph{non-decreasing} stepsize rule for the monotone variational inclusion $0\in F(x)+\partial g(x)$, where $F:\mathbb{H}\to\mathbb{H}$ is monotone and $L$-Lipschitz and $g:\mathbb{H}\to\mathbb{R}\cup\{\infty\}$ is proper, convex, and lower semicontinuous; this setting reduces to \eqref{main} by taking $A=\partial g$ and $B=F$. Building on this idea, Yao et al.~\cite{yao2025forward} proposed a double-inertial forward--reflected--backward scheme in which the Lipschitz behavior of the cost operator is estimated on the fly via a non-decreasing stepsize sequence. A key benefit of non-decreasing stepsizes is their ability to recover from conservative initialization, as the steps grow, they can enter a stable and effective regime, thereby improving the overall efficiency and adaptability of the method.

The main contributions of this paper are summarized as follows:
\begin{itemize}
    \item Motivated by the above discussions on the advantages of non-decreasing step-sizes over decaying step-sizes, we establish weak convergence of the generalized forward–reflected–backward (GFRB) method \eqref{GFRB} under the step-size rule proposed by Hoai \cite{hoai2025new} for solving the monotone inclusion \eqref{main}. This step-size eliminates the need to compute the Lipschitz constant of $B$ to run the algorithm.
    \item We further investigate GFRB~\eqref{GFRB} in the special case where $A \equiv  0$ and $B = \begin{bmatrix} 0 & -1 \\[2pt] 1 & 0 \end{bmatrix}$ on $\mathbb{H} = \mathbb{R}^2$, 
    and establish that its convergence rate cannot exceed $\frac{1}{\sqrt{2}}$. Furthermore, for the case $A \equiv 0$ and $B = I$ on $\mathbb{H} = \mathbb{R}^n$, we demonstrate a better linear rate of convergence for GFRB than some of the existing methods, such as FBF and FRB.
    \item We derive a new extended version of the primal-dual twice reflected algorithm (PDTR) from the GFRB to solve a more general structured three-operator inclusion problem with proper choices of metric.
    \item We finally apply the proposed step-size strategy to several benchmark optimization problems and demonstrate its superiority over existing methods.
\end{itemize}

\section{Preliminaries}\label{sec:prelim}
Throughout, $(\mathbb{H},\langle\cdot,\cdot\rangle)$ denotes a real Hilbert space with the induced norm $\|x\|:=\sqrt{\langle x,x\rangle}$. 
We write $I$ for the identity on $\mathbb{H}$.
Given a set-valued operator $A:\mathbb{H}\rightrightarrows\mathbb{H}$, its graph and domain are defined as
$\operatorname{gra}A:=\{(x,u)\in\mathbb{H}\times\mathbb{H}: u\in Ax\}$ and $\operatorname{dom}A:=\{x: Ax\neq\varnothing\}$, respectively. The inverse of $A$ is $A^{-1}(u):=\{x:\,u\in Ax\}$, and the resolvent of $A$ with a parameter $\lambda>0$ is 
\[
J_{\lambda A}:=(I+\lambda A)^{-1}.
\]
When $A$ is maximally monotone, the resolvent $J_{\lambda A}$ of $A$ is single-valued. An operator $A:\mathbb{H}\rightrightarrows\mathbb{H}$ is \emph{monotone} if 
\[
\langle u-v,x-y\rangle\ge 0\quad\forall (x,u),(y,v)\in\operatorname{gra}A,
\]
and \emph{maximally monotone} if $\operatorname{gra}A$ has no proper monotone extension, that is, for any monotone operator $\bar A$ on $\mathbb{H}$, $\operatorname{gra}\bar A\subset\operatorname{gra}A$. An operator $B:\mathbb{H}\to\mathbb{H}$ is called \emph{$L$-Lipschitz} if 
$$
\|Bx-By\|\le L\|x-y\|~~\forall x,y\in\mathbb{H}.
$$
Next, we recall some useful lemmas that will be helpful for the convergence analysis of our proposed algorithm.
\begin{lemma}\label{eq:cosrule}
    Let $a,b,c\in\mathbb{H}$. Then, 
    \begin{equation}
2\langle a-b,c-a\rangle=\|b-c\|^2-\|a-c\|^2-\|a-b\|^2.
\end{equation}
\end{lemma}
\begin{lemma} \cite{brezis1973operateurs}\label{lem:sum_max_monotone}
    If $A:\mathbb{H}\rightrightarrows\mathbb{H}$ is a maximal monotone operator and $B:\mathbb{H}\rightrightarrows\mathbb{H}$ is monotone and Lipschitz continuous, then $A + B$ is maximally monotone.
\end{lemma}
\begin{lemma} (Opial) \cite{opial1967weak}\label{lem:Opial}
Let $(z_k)$ be a sequence in $\mathbb{H}$ and $S\subset\mathbb{H}$ is nonempty. If
\begin{enumerate}
    \item for every $s\in S$, the limit $\lim_{k\to\infty}\|z_k-s\|$ exists and
    \item every weak cluster point of $(z_k)$ belongs to $S$, 
\end{enumerate}
then $(z_k)$ converges weakly to a point in $S$.
\end{lemma}

\section{GFRB with a non-decreasing step-size strategy}
In this section, we analyze GFRB \eqref{GFRB} with a non-decreasing step-size strategy proposed in~\cite{hoai2025new}. A key advantage of this strategy is that it dispenses with any
a priori estimate of the Lipschitz constant $L$ of $B$ while ensuring that the step-sizes do not decay to zero.

\begin{algorithm}[ht]
\caption{GFRB with non-decreasing step-size to solve \eqref{main}}
\label{alg:forb_incr}
\small
\begin{algorithmic}[1]
\State \textbf{Initialize:} Let $x_{-1},x_{0}, x_1\in \mathbb{H}$, $0<c_{1}<c_{2}<\frac{1-\varepsilon-\alpha}{2|\delta|+2}$, where $\varepsilon>0$, and $0\le \alpha<1$. Let
$\lambda_{-1},\lambda_{0}>0$, and $(\gamma_k)_{k\ge0}\subset\mathbb{R}_+$ is a sequence with $\sum_{k=0}^\infty \gamma_k<\infty$.
\For{$k=1,2,\dots$}
  \State \textbf{Step-size Update:}
  \Statex \hspace*{\algorithmicindent}
  \begin{equation}\label{eq:gamma_update}
    \lambda_{k}
    =\begin{cases}
      \dfrac{c_{1}\,\|x_{k-1}-x_{k}\|}{\|B x_{k-1}-B x_{k}\|},
        &\text{if }\|B x_{k-1}-B x_{k}\|>\tfrac{c_{2}}{\lambda_{k-1}}\|x_{k-1}-x_{k}\|,\\[2pt]
      \left(1+\gamma_{k-1}\right)\,\lambda_{k-1}, &\text{otherwise.}
    \end{cases}
  \end{equation}
\State \textbf{Compute:}
  \Statex \hspace*{\algorithmicindent}
  \begin{equation}\label{eq:x_k_update}
    x_{k+1}
    =J_{\lambda_{k}A}\!\left(
    \begin{aligned}
      &(1-\alpha)x_{k}+\alpha x_{k-1}-\lambda_{k}B(x_{k})\\
      &\quad-\lambda_{k-1}(1+\delta)\bigl(B(x_{k})-B(x_{k-1})\bigr)
      +\lambda_{k-2}\delta\bigl(B(x_{k-1})-B(x_{k-2})\bigr)
    \end{aligned}
    \right).
  \end{equation}
\EndFor
\end{algorithmic}
\end{algorithm}

In practice, we choose the value of $\varepsilon$ to be close to zero. Furthermore, note that if $\delta=\alpha=0$ in Algorithm \ref{alg:forb_incr}, then GFRB reduces to FRB and thus can also be seen as a new step-size strategy for the FRB algorithm.

\begin{lemma}\cite[Lemma 3.2]{hoai2025new}\label{Lemma_pham_Thi_Hoai}
Let $\{\lambda_k\}$ be the step-size sequence generated by Algorithm \ref{alg:forb_incr}. Then
\begin{enumerate}
    \item $\{\lambda_k\}$ is convergent and $\forall k \geq 1$, we have $\lambda_k \geq \lambda_{\min} = \min \left\{ \frac{c_1}{L}, \lambda_0 \right\} > 0$;
    \item there exists a positive integer $\tilde k$ such that $\lambda_{k+1} > \lambda_k$ $\forall k \geq \tilde k$.
\end{enumerate}
\end{lemma}
\begin{lemma}\label{lem:first}
Let $A:\mathbb{H}\rightrightarrows \mathbb{H}$ be maximally monotone and let $B:\mathbb{H}\to\mathbb{H}$ be monotone and $L$-Lipschitz. 
Assume that $x\in (A+B)^{-1}(0)$ and let $\{(x_k,\lambda_k)\}$ be generated by Algorithm~\ref{alg:forb_incr}. Then there exists a natural number $k_3$ such that, for all $k\ge k_3$,
\begin{equation}\label{eq:lem_descent}
S_{k+1}\;\le\; S_k-\varepsilon\|x_k-x_{k-1}\|^2,
\end{equation}
where
\begin{align}
S_k
&:= \|x_k-x\|^2
+2\lambda_{k-1}\big\langle B(x_k)-B(x_{k-1}),\,x-x_k\big\rangle \nonumber\\
&\quad+\bigl(1-\alpha-c_2-|\delta|c_2\bigr)\|x_k-x_{k-1}\|^2
+2\delta\lambda_{k-2}\big\langle B(x_{k-1})-B(x_{k-2}),\,x_k-x\big\rangle \\
&\quad+ \alpha\|x_{k-1}-x\|^2
+|\delta|c_2\|x_{k-1}-x_{k-2}\|^2.
\label{eq:def_Sk}
\end{align}
\end{lemma}

\begin{proof}
Since $A$ is monotone, from \eqref{eq:x_k_update}, we have
\begin{equation}\label{ineq:main}
\begin{aligned}
 \Big\langle\,
& (1-\alpha)x_k +\alpha x_{k-1}- x_{k+1} - \lambda_k B(x_k)
  - \lambda_{k-1}(1+\delta)\bigl(B(x_k)-B(x_{k-1})\bigr) \\
& {}+ \delta\lambda_{k-2}\bigl(B(x_{k-1})-B(x_{k-2})\bigr)
  + \lambda_k Bx,\; x_{k+1}-x
\Big\rangle \geq 0.
\end{aligned}
\end{equation}
This can be equivalently expressed as
\begin{multline}\label{eq: monotne_ineq}
    \langle x_k-x_{k+1}, x_{k+1}-x\rangle -\alpha \langle x_k-x_{k-1}, x_{k+1}-x_k\rangle -\alpha \langle x_k-x_{k-1}, x_{k}-x\rangle\\
    - \lambda_k\langle B(x_k)-B(x), x_{k+1}-x\rangle - \lambda_{k-1}(1+\delta)\langle B(x_k)-B(x_{k-1}), x_{k+1}-x\rangle\\
    +\delta\lambda_{k-2}\langle B(x_{k-1})-B(x_{k-2}), x_{k+1}-x\rangle\geq0.
\end{multline}
The following terms from \eqref{eq: monotne_ineq} can be rewritten as
\begin{align}\label{eq: expan_1}
    &\langle B(x_k)-B(x), x_{k+1}-x\rangle\nonumber \\&
    = \langle B(x_k)-B(x_{k+1}), x_{k+1}-x\rangle + \langle B(x_{k+1})-B(x),\,x_{k+1}-x\rangle\nonumber\\&
    \geq \langle B(x_k)-B(x_{k+1}), x_{k+1}-x\rangle ~~\text{by monotonicity of}~~B.
\end{align}
Additionally, we have
\begin{align}\label{eq: expan_2}
    &(1+\delta)\langle B(x_k)-B(x_{k-1}), x_{k+1}-x\rangle\nonumber\\&
    = \delta\langle B(x_k)-B(x_{k-1}), x_{k+1}-x\rangle + \langle B(x_k)-B(x_{k-1}), x_{k+1}-x_k\rangle\nonumber\\&
    +\langle B(x_k)-B(x_{k-1}), x_k-x\rangle,
\end{align}
and 
\begin{align}\label{eq: expan_3}
    &\delta\langle B(x_{k-1})-B(x_{k-2}), x_{k+1}-x\rangle\nonumber\\&
    = \delta\langle B(x_{k-1})-B(x_{k-2}), x_{k+1}-x_k\rangle+\delta\langle B(x_{k-1})-B(x_{k-2}), x_k-x\rangle.
\end{align}
Furthermore, the first three terms in \eqref{eq: monotne_ineq} can be estimated using the Cauchy--Schwarz inequality and Lemma~\ref{eq:cosrule} as follows
\begin{align}
\alpha\big\langle x_k-x_{k-1},\,x_{k+1}-x_k\big\rangle
&\le \alpha\|x_k-x_{k-1}\|\,\|x_{k+1}-x_k\|
\nonumber\\
&\le \frac{\alpha}{2}\Big(\|x_k-x_{k-1}\|^2+\|x_{k+1}-x_k\|^2\Big).
\label{eq:cs_young_term}
\\[0.8em]
\big\langle x_k-x_{k+1},\, x_{k+1}-x\big\rangle
&=\frac12\Big(\|x_k-x\|^2-\|x_{k+1}-x\|^2-\|x_{k+1}-x_k\|^2\Big),
\nonumber\\
\big\langle x_k-x_{k-1},\, x_k-x\big\rangle
&=\frac12\Big(\|x_k-x_{k-1}\|^2+\|x_k-x\|^2-\|x_{k-1}-x\|^2\Big).
\label{eq:cosrule_terms}
\end{align}
By combining \eqref{eq: expan_1}, \eqref{eq: expan_2}, \eqref{eq: expan_3}, \eqref{eq:cs_young_term} and \eqref{eq:cosrule_terms} with \eqref{eq: monotne_ineq}, we obtain
\begin{multline}\label{Theorem1_eq:1}
\|x_{k+1} - x\|^2 + 2\lambda_k \langle B(x_k) - B(x_{k+1}), x_{k+1}-x\rangle + (1-\alpha) \|x_{k+1} - x_k\|^2\\
+2\delta\lambda_{k-1}\langle B(x_k)-B(x_{k-1}), x_{k+1}-x\rangle \leq (1-\alpha)\|x_k - x\|^2 + \alpha\|x_{k-1}-x\|^2\\
+ 2\lambda_{k-1} \langle B(x_k) - B(x_{k-1}), x - x_k \rangle  + 2\delta\lambda_{k-2} \langle B(x_{k-1}) - B(x_{k-2}), x_k - x\rangle \\+2\delta\lambda_{k-2}\langle B(x_{k-1})-B(x_{k-2}), x_{k+1}-x_k\rangle
+2\lambda_{k-1}\langle B(x_{k-1})-B(x_k), x_{k+1}-x_k\rangle.
\end{multline}
Now, if $\|Bx_{k-1} - Bx_k\| > \frac{c_2}{\lambda_{k-1}} \|x_{k-1} - x_k\|$ holds, then using $\lambda_k = \frac{c_1\|x_{k-1} - x_k\|}{\|Bx_{k-1} - Bx_k\|}$ and the Cauchy--Schwarz inequality, we have
\begin{equation}\label{Lemma-1_eq:3}
 \begin{split}
\lambda_{k-1}\langle B(x_k) - B(x_{k-1}), x_k - x_{k+1} \rangle &\leq\frac{c_1\lambda_{k-1}}{\lambda_k} \|x_k - x_{k-1}\| \|x_k - x_{k+1}\|\\
& \leq \frac{c_1\lambda_{k-1}}{2\lambda_k}\|x_k-x_{k-1}\|^2 + \frac{c_1\lambda_{k-1}}{2\lambda_k}\|x_{k+1}-x_k\|^2.
\end{split}
\end{equation}
By Lemma \ref{Lemma_pham_Thi_Hoai}, we have $\lim_{k\to\infty} \frac{c_1\lambda_{k-1}}{\lambda_k}=c_1<c_2$. Thus, there exists a natural number $k_1$ such that 
\begin{equation}\label{ineq_1}
    \frac{c_1\lambda_{k-1}}{\lambda_k}\leq c_2 ~~\forall k\geq k_1.
\end{equation}
Furthermore, when  $\|Bx_{k-1} - Bx_k\| \leq \frac{c_2}{\lambda_{k-1}} \|x_{k-1} - x_k\|$, following the Cauchy--Schwarz inequality, we obtain
\begin{equation}\label{Lemma-1_eq:4}
 \begin{split}
\lambda_{k-1}\langle B(x_k) - B(x_{k-1}), x_k - x_{k+1} \rangle &\leq\lambda_{k-1}\|Bx_k - Bx_{k-1}\| \|x_k - x_{k+1}\|\\
& \leq \frac{c_2}{2}\|x_k-x_{k-1}\|^2 + \frac{c_2}{2}\|x_{k+1}-x_k\|^2.
\end{split}
\end{equation}
Therefore, from \eqref{Lemma-1_eq:3},  \eqref{Lemma-1_eq:4} and Lemma \ref{Lemma_pham_Thi_Hoai}, we finally obtain
\begin{equation}\label{Lemma-1_eq:5}
    \lambda_{k-1}\langle B(x_k) - B(x_{k-1}), x_k - x_{k+1} \rangle\leq \frac{c_2}{2}\|x_k-x_{k-1}\|^2 + \frac{c_2}{2}\|x_{k+1}-x_k\|^2~\text{for all}~ k\geq k_1.
\end{equation}
Similarly, there exists another natural number $k_2$ such that for all $k\geq k_2$, we have
\begin{equation}\label{Lemma-1_eq:7}
    \delta\lambda_{k-2}\langle B(x_{k-1}) - B(x_{k-2}), x_{k+1} - x_k \rangle\leq \frac{|\delta| c_2}{2}\|x_{k-1}-x_{k-2}\|^2 + \frac{|\delta| c_2}{2}\|x_{k+1}-x_k\|^2.
\end{equation}
Let $k_3=\max\{k_1,k_2\}$. Then, for all $k\geq k_3$, combining \eqref{Theorem1_eq:1}, \eqref{Lemma-1_eq:5}, and \eqref{Lemma-1_eq:7}, we obtain
\begin{multline*}
\|x_{k+1} - x\|^2 + 2\lambda_k \langle B(x_{k+1}) - B(x_k), x - x_{k+1} \rangle + (1-\alpha-c_2-|\delta| c_2)\|x_{k+1} - x_k\|^2\\
+2\delta\lambda_{k-1}\langle B(x_k)-B(x_{k-1}), x_{k+1}-x\rangle \leq (1-\alpha) \|x_k - x\|^2 +\alpha\|x_{k-1}-x\|^2
+ 2\lambda_{k-1} \langle B(x_k) - B(x_{k-1}), x - x_k \rangle \\
+c_2\|x_k-x_{k-1}\|^2 + |\delta| c_2\|x_{k-1}-x_{k-2}\|^2
+ 2\delta\lambda_{k-2} \langle B(x_{k-1}) - B(x_{k-2}), x_k - x\rangle.
\end{multline*}
Noting that $0<c_{2}<\frac{1-\alpha-\varepsilon}{2|\delta|+2}$ implies $0<\varepsilon+|\delta|c_2+c_2\leq1-\alpha-c_2-|\delta|c_2$, it follows Lemma~\ref{lem:first}.
\end{proof}
\begin{theorem}\label{x_k is bdd lemma}
    Let $A\colon\mathbb{H}\rightrightarrows\mathbb{H}$ be maximally monotone and $B:\mathbb{H}\to\mathbb{H}$ be $L$-Lipschitz and monotone. Assume that $(A+B)^{-1}(0)\neq\varnothing$ and let $\{(x_k,\lambda_k)\}$ be the sequence generated by Algorithm~\ref{alg:forb_incr}. Then,  $(x_k)$ converges weakly to a solution of \eqref{main}.
\end{theorem}
\begin{proof}
  First, observe that
\begin{align*}
S_k
&:= \|x_k-x\|^2
+2\lambda_{k-1}\big\langle B(x_k)-B(x_{k-1}),\,x-x_k\big\rangle \nonumber\\
&\quad+\bigl(1-\alpha-c_2-|\delta|c_2\bigr)\|x_k-x_{k-1}\|^2
+2\delta\lambda_{k-2}\big\langle B(x_{k-1})-B(x_{k-2}),\,x_k-x\big\rangle \\
&\quad+ \alpha\|x_{k-1}-x\|^2
+|\delta|c_2\|x_{k-1}-x_{k-2}\|^2.
\label{eq:def_Sk}
\end{align*}
From the arguments as in \eqref{Lemma-1_eq:5} and \eqref{Lemma-1_eq:7}, there exist two natural numbers $k_4$ and $k_5$ such that 
\begin{equation}\label{Lemma-1_eq:8}
    \lambda_{k-1}\langle B(x_k) - B(x_{k-1}), x - x_k \rangle\geq -\frac{c_2}{2}\|x_k-x_{k-1}\|^2 - \frac{c_2}{2}\|x-x_k\|^2~\text{for all}~ k\geq k_4,
\end{equation}
and 
\begin{equation}\label{Lemma-1_eq:9}
    \delta\lambda_{k-2}\langle B(x_{k-1}) - B(x_{k-2}), x_k - x \rangle\geq -\frac{|\delta| c_2}{2}\|x_{k-1}-x_{k-2}\|^2 - \frac{|\delta| c_2}{2}\|x_k-x\|^2~\text{for all}~ k\geq k_5.
\end{equation}
Let $k\geq k_6 =\max\{k_3,k_4, k_5\}$. Then, from \eqref{eq:def_Sk}, \eqref{Lemma-1_eq:8}, \eqref{Lemma-1_eq:9} and noting that $0<\varepsilon+|\delta|c_2+c_2\leq1-\alpha-c_2-|\delta|c_2$, we obtain
\begin{equation*} 
\begin{aligned}
    S_k &\geq (1 - c_2 - |\delta| c_2)\|x_k - x\|^2 +(\varepsilon+|\delta c_2|)\|x_k-x_{k-1}\|^2+\alpha\|x_{k-1}-x\|^2\\
        &\geq (\varepsilon+|\delta| c_2 + c_2)\|x_k - x\|^2+(\varepsilon+|\delta c_2|)\|x_k-x_{k-1}\|^2+\alpha\|x_{k-1}-x\|^2\\
        &\geq 0.
\end{aligned}
\end{equation*}
Thus, $(S_k)$ is non-negative, and from \eqref{eq:lem_descent} it follows that the sequence $(S_k)$ converges and 
$$
\lim_{k \to \infty} \|x_{k+1} - x_k\| = 0.
$$
Furthermore, by taking into account the Lipschitz continuity of $B$ and $\lim_{k\to\infty}\|x_{k+1}-x_k\|=0$ together with \eqref{eq:def_Sk}, we obtain
\begin{equation}\label{lim S_k}
    \lim_{k\to\infty} S_k = \lim_{k\to\infty} \|x_k-x\|^2 +\alpha\|x_{k-1}-x\|^2.
\end{equation}
This implies that $(x_k)$ is a bounded sequence. Hence, there exists a weakly convergent subsequence $(x_{k_n})$ such that $(x_{k_n})$ converges weakly to $\bar x\in\mathbb{H}$.
From \eqref{eq:x_k_update}, we have
\begin{equation}\label{main_lemma_1}
\begin{aligned}
    &\frac{1}{\lambda_{k-1}} \bigg( 
        (1-\alpha)x_{k-1} - x_k  + \alpha x_{k-2}
        + \lambda_{k-1} \big( B(x_k) - B(x_{k-1}) \big) \\
    &\qquad\quad + \lambda_{k-2}(1 + \delta) \big( B(x_{k-2}) - B(x_{k-1}) \big) 
        + \delta \lambda_{k-3} \big( B(x_{k-2}) - B(x_{k-3}) \big) 
    \bigg)\in(A+B)(x_k).
\end{aligned}
\end{equation}
Using the fact that $(\lambda_k)$ is convergent (by Lemma~\ref{Lemma_pham_Thi_Hoai}),  
$\lim_{k\to\infty} \|x_{k+1} - x_k\| = 0$, and the Lipschitz continuity of $B$,  
we deduce that the LHS of \eqref{main_lemma_1} converges strongly to $0$.  
Moreover, by Lemma~\ref{lem:sum_max_monotone}, the operator $A+B$ is maximally monotone;  
hence its graph is sequentially closed in $\mathbb{H}^{\text{weak}} \times \mathbb{H}^{\text{strong}}$.  
Since $(x_{k_n})$ converges weakly to $\bar{x}$, it follows from \eqref{main_lemma_1} that $0 \in (A+B)(\bar{x})$. Furthermore, by \eqref{lim S_k}, the limit $\lim_{k\to\infty} \|x_k - \bar{x}\|$ is finite. As the limit point $\bar{x}$ was arbitrary, Lemma~\ref{lem:Opial} implies that $(x_k)$  
converges weakly. This completes the proof.
\end{proof}

\section{Rate of Convergence}
In this section, we look at a simple concrete example and compare the linear rate of GFRB~\eqref{GFRB} with that of FRB~\eqref{FRB}. In particular, we show that, on this example, GFRB cannot have a better convergence rate than FRB: the rate is bounded below by $\tfrac{1}{\sqrt{2}}$.

We take $\mathbb{H}=\mathbb{R}^2$ and consider
\[
A \equiv 0,
\qquad
B=\begin{bmatrix}0&-1\\[2pt]1&0\end{bmatrix}.
\]
Fix $\alpha=0$ in \eqref{GFRB}. Then \eqref{GFRB} can be rewritten as a fixed-point iteration of the form
\[
u_{k+1}=Mu_k,
\qquad
u_k=\begin{bmatrix}x_k\\ x_{k-1}\\ x_{k-2}\end{bmatrix},
\]
with
\begin{equation}\label{eq:fix point iter matrix}
M=
\begin{bmatrix}
I-\lambda(\delta+2)B & \lambda(2\delta+1)B & -\lambda\delta B\\
I & 0 & 0\\
0 & I & 0
\end{bmatrix},
\qquad \delta\in\mathbb{R}.
\end{equation}
When $\delta=0$ (so that \eqref{GFRB} reduces to FRB in this setting), Malitsky and Tam~\cite[Section~2]{malitsky2020forward} showed that the spectral radius of $M$ is exactly $\tfrac{1}{\sqrt{2}}$. Our goal here is to show that as long as $\delta\neq 0$, the spectral radius of the same matrix $M$ stays \emph{above} $\tfrac{1}{\sqrt{2}}$. In particular, on this example, introducing a nonzero $\delta$ cannot yield a faster linear rate than FRB.

\begin{lemma}\label{lem:rate_lower_bound}
Consider \eqref{main} with $\mathbb{H}=\mathbb{R}^2$, $A\equiv 0$, and $B=\begin{bmatrix}0&-1\\[2pt]1&0\end{bmatrix}.$ Let $\alpha=0$ and write the GFRB iteration \eqref{GFRB} in the fixed-point form
\[
u_{k+1}=Mu_k,
\qquad 
u_k=\begin{bmatrix}x_k\\ x_{k-1}\\ x_{k-2}\end{bmatrix},
\]
where $M$ is given in \eqref{eq:fix point iter matrix}. Then, for every $\delta\in\mathbb{R}\setminus\{0\}$,
\[
\rho(M)>\frac{1}{\sqrt{2}},
\]
and equality holds when $\delta=0$.
\end{lemma}

\begin{proof}
In the GFRB algorithm \eqref{GFRB}, for any $\delta\in\mathbb{R}$ with the step-size $\lambda \approx \frac{1}{2|\delta|+2}$, the characteristic polynomial of $M$ can be written as
\begin{equation}\label{ch.poly}
    \xi_{M}(z) = (z^3-z^2)^2 +\frac{1}{(2|\delta|+2)^2}\left[(\delta+2)z^2-(2\delta+1)z+\delta\right]^2.
\end{equation}
From \eqref{ch.poly}, observe that $\xi_{M}(z) = p_{\pm i}(z)$, where 
$$
p_{\pm i}(z)=(z^3-z^2) \pm \frac{i}{(2|\delta|+2)^2}\left[(\delta+2)z^2-(2\delta+1)z+\delta\right].
$$
Since $\delta\in\mathbb{R}$, all the roots of $\xi_{M}(z)$ are the roots of one of the $p_{\pm i}(z)$ and the corresponding conjugate pairs. Without loss of generality, we work with $p_{i}$. From this point, we consider two cases for $\delta$. When $\delta>0$, $p_{i}(z)$ is given by
\begin{equation}\label{p_i(z)}
    p_{i}(z)=(z^3-z^2) + \frac{i}{(2\delta+2)^2}\left[(\delta+2)z^2-(2\delta+1)z+\delta\right].
\end{equation}
Now, our goal is to show that for any $\delta>0$, at least one root of $p_{i}(z)$ has modulus greater than $\frac{1}{\sqrt{2}}$, and hence $\xi_{M}(z)$ will have at least a root of modulus greater than $\frac{1}{\sqrt{2}}$. This, in fact, will imply that $\rho(M)>\frac{1}{\sqrt{2}}$ whenever $\delta>0$. We notice that \eqref{p_i(z)} can be written as
\begin{equation*}
    p_i(z)\;=\;z^3-\Bigl(1-\alpha(\delta+2)\Bigr)z^2-\alpha(2\delta+1)\,z+\alpha\,\delta,~~~\text{where}~~\alpha=\frac{i}{2(\delta+1)}.
\end{equation*}
Set $r:=1/\sqrt{2}$ and define
\[
q(z)\;:=\;\frac{1}{r^3}\,p_i(rz)=z^3+a z^2+b z+c,
\]
where, using $r^{-3}=2\sqrt{2}$ and $\alpha=\frac{i}{2(\delta+1)}$, one has
\begin{align*}
a&=\frac{-1+\alpha(\delta+2)}{r}
   =-\sqrt{2}+\frac{i\sqrt{2}(\delta+2)}{2(\delta+1)},\\
b&=\frac{-\alpha(2\delta+1)}{r^2}
   =-\frac{i(2\delta+1)}{\delta+1}~~\text{and}\\
c&=\frac{\alpha\,\delta}{r^3}
   =\frac{i\sqrt{2}\,\delta}{\delta+1}.
\end{align*}
By construction, $p_i$ has a zero with $|z|>r$ if and only if $q$ has a zero with $|z|>1$, and $|z|=r$ for $p_i$ corresponds to $|z|=1$ for $q$. Now, by Schur--Cohn theorem \cite[Theorem 6.8b]{henrici1993applied}, $q$ has all zeros in the closed unit disk $D= \{z:|z|\leq1\}$ if and only if $D_1>0$, $D_2>0$, where 
\begin{equation}
    D_1:=1-|c|^2,\qquad D_2:=(1-|c|^2)^2-|a-b\overline{c}|^2.
\end{equation}
 For $a,b,c$ as defined above with $\delta>0$, an elementary calculation yields
\begin{equation}\label{Schur-Cohn-1}
\begin{aligned}
D_1 &= 1-|c|^2 = 1-\frac{2\delta^2}{(\delta+1)^2} = \frac{-\delta^2+2\delta+1}{(\delta+1)^2}, \\
D_2 &= (1-|c|^{2})^{2}-|a-b\overline{c}|^{2}
=-\frac{3\delta^{4}+6\delta^{3}+5\delta^{2}+12\delta+6}{2(\delta+1)^{4}}.
\end{aligned}
\end{equation}
From \eqref{Schur-Cohn-1}, notice that $D_1=1-|c|^2>0$ if and only if $0<\delta<\sqrt{2}+1$ and $0$ at $\delta =\sqrt{2}+1$. However, $D_2<0$ for all $\delta>0$. Therefore, by Schur--Cohn theorem \cite[Theorem 6.8b]{henrici1993applied}, as long as $\delta>0$, $q$ has a zero outside the closed unit disk $D$, and hence $p_i$ has a zero of modulus greater than $\frac{1}{\sqrt{2}}$.\\
For the other case $\delta<0$, by a similar approach, from \eqref{ch.poly}, we obtain
\begin{equation*}
    p_{i}(z)=(z^3-z^2) + \frac{i}{(2\delta-2)^2}\left[(\delta+2)z^2-(2\delta+1)z+\delta\right],
\end{equation*}
and $q(z)=z^3+a z^2+b z+c$, where using $\alpha=\frac{i}{2(1-\delta)}$, the values of $a,b,c$ are given by
\begin{align*}
a&=\frac{-1+\alpha(2-\delta)}{r}
   =-\sqrt2+\frac{i\sqrt2(\delta+2)}{2(1-\delta)},\\
b&=-\frac{\alpha(2\delta+1)}{r^{2}}
   =-\frac{i(2\delta+1)}{1-\delta},\\
c&=\frac{\alpha\delta}{r^{3}}
   =\frac{i\sqrt2\,\delta}{1-\delta}.
\end{align*}
Moreover, see that, $D_1=1-|c|^2>0$ if and only if $-\sqrt{2}-1<\delta<0$. However, $D_2 =-\dfrac{3\delta^{4}-6\delta^{3}+5\delta^{2}-12\delta+6}{2(\delta-1)^{4}}<0$ as long as $\delta<0$. Thus both $D_1>0$ and $D_2>0$ do not hold simultaneously whenever $\delta<0$. Therefore, in both cases whenever $\delta\in\mathbb{R}\setminus\{0\}$, $p_i(z)$ has at least a zero inside the disk $\bar D=\left\{z: |z|>\frac{1}{\sqrt{2}}\right\}$. Furthermore, at $\delta=0$, notice that
$$
p_i(z)=z\Bigl(z-\tfrac{1-i}{2}\Bigr)^2.
$$
Therefore, all the roots of $\xi_{M}$ are $0,0, \frac{1+i}{2}, \frac{1+i}{2}, \frac{1-i}{2}, \frac{1-i}{2}$. Hence, the rate of convergence is $\sqrt{\frac{1}{4}+\frac{1}{4}}=\frac{1}{\sqrt{2}}$. This completes the proof.
\end{proof}

The following table (Table \ref{tab:eigenvalues}) shows the rates of convergence for different choices of $\lambda$ and $\delta$ values. 
\begin{table}[ht]
\centering
\caption{Spectral radius $\rho(M)$ for different $\lambda$ and $\delta$ values.}
\begin{tabular}{|c|c|c|c|}
\hline
\textbf{$\delta$} & \textbf{ROC with $\left(\lambda = \frac{1}{2\delta+2}\right)$} & \textbf{ROC with $\left(\lambda = \frac{1}{3\delta+3}\right)$} & \textbf{ROC with $\left(\lambda = \frac{1}{2\delta+3}\right)$} \\
\hline
0 & 0.707107 (MT \cite{malitsky2020forward}) & 0.710023 & 0.710023 \\
0.0020 & 0.708176 & 0.710619 & 0.710604 \\
0.0121 & 0.713446 & 0.713627 & 0.713529 \\
0.0141 & 0.714484 & 0.714234 & 0.714117 \\
0.0162 & 0.715518 & 0.714843 & 0.714707 \\
0.0303 & 0.722608 & 0.719148 & 0.718864 \\
0.0323 & 0.723600 & 0.719768 & 0.719461 \\
0.0343 & 0.724587 & 0.720390 & 0.720059 \\
0.0364 & 0.725569 & 0.721013 & 0.720658 \\
0.0404 & 0.727516 & 0.722261 & 0.721858 \\
0.0525 & 0.733234 & 0.726028 & 0.725470 \\
0.0545 & 0.734169 & 0.726659 & 0.726074 \\
0.0566 & 0.735098 & 0.727290 & 0.726678 \\
0.0586 & 0.736023 & 0.727921 & 0.727283 \\
0.0606 & 0.736943 & 0.728554 & 0.727887 \\
0.0626 & 0.737857 & 0.729186 & 0.728492 \\
0.0667 & 0.739671 & 0.730453 & 0.729703 \\
0.0687 & 0.740571 & 0.731087 & 0.730309 \\
0.0808 & 0.745865 & 0.734895 & 0.733945 \\
0.0909 & 0.750144 & 0.738072 & 0.736974 \\
0.0929 & 0.750986 & 0.738707 & 0.737580 \\
0.0949 & 0.751823 & 0.739342 & 0.738185 \\
\hline
\end{tabular}
\label{tab:eigenvalues}
\end{table}

\subsection{Improved Rate of Convergence}
In this subsection we focus on another instance of \eqref{main} when $\mathbb{H}=\mathbb{R}^n$ with $A\equiv 0$ and $B=I$. 
We illustrate the rates of convergence of GFRB~\eqref{GFRB} for different choices of $\lambda$, $\delta$ with $\alpha=0$, and compare them with other existing methods.

\begin{enumerate}
    \item \textbf{Tseng's FBF method \cite{tseng2000modified}}: In the present setting, Tseng's FBF method iterates as 
    $$
    x_{k+1}=(1-\lambda+\lambda^2)x_k.
    $$
    It is not hard to see that the optimal value of $\lambda$ is $\frac{1}{2}$, which gives a rate $\frac{3}{4}$.
    \item \textbf{Forward-reflected-backward method \cite{malitsky2020forward}}: In this case, the iteration scheme of FRB with fixed step-size $\lambda\in(0,\frac{1}{2})$ is 
    $$
     x_{k+1} = (1-2\lambda)x_k + \lambda x_{k-1}. 
    $$
    Letting $\lambda\approx\frac{1}{2}$ yields the rate closer to $\frac{1}{\sqrt{2}}$.
    \item \textbf{GFRB \cite[Example 3.1]{dung2025generalization}}: Choosing $\alpha=0$, $\lambda=\frac{2}{15}$ and $\delta=\frac{3}{2}$ in \eqref{GFRB}, the iteration becomes 
    $$
    x_{k+1}=\frac{8}{15}x_k+\frac{8}{15}x_{k-1}-\frac{1}{5}x_{k-2}. 
    $$
    For $x_0\in\mathbb{H}$, $x_1=\tfrac{x_0}{3}$ and $x_2=\tfrac{x_0}{3^2}$, one has $x_{k+1}=\tfrac{x_0}{3^k}$ for all $k\ge1$, hence the rate of convergence is $\frac{1}{3}$.
    \item \textbf{Improved rates for GFRB}:
    \begin{enumerate}
        \item Let $\alpha=0$, $\lambda=\frac{68}{285}$ and $\delta=\frac{27}{68}$. Then, \eqref{GFRB} iterates as
        \begin{equation*}
             x_{k+1}=\frac{122}{285}x_k+\frac{122}{285}x_{k-1}-\frac{27}{285}x_{k-2}.
        \end{equation*}
If we assume the initial points as $x_0\in\mathbb{H}$, $x_1=\frac{x_0}{5}$, and $x_2=\frac{x_0}{5^2}$, then $x_{k+1}=\frac{x_0}{5^k}$. Thus, the rate of convergence is $\frac{1}{5}<\frac{1}{3}<\frac{1}{\sqrt{2}}<\frac{3}{4}$. So, in this particular problem, we show that GFRB \eqref{GFRB} is faster than the aforementioned methods.
        \item Furthermore, if we take $\alpha=0$, $\lambda=\frac{45}{174}$ and $\delta=\frac{13}{45}$ with initial points as $x_0\in\mathbb{H}$, $x_1=\frac{x_0}{6}$, and $x_2=\frac{x_0}{6^2}$, then from \eqref{GFRB}, we have
        \begin{equation}\label{roc 1/6}
            x_{k+1}=\frac{71}{174}x_k+\frac{71}{174}x_{k-1}-\frac{13}{174}x_{k-2}.
        \end{equation}
        Iterating \eqref{roc 1/6}, we obtain $x_{k+1}=\frac{x_0}{6^k}$. Therefore, we have the rate $\frac{1}{6}$.
    \end{enumerate}
\end{enumerate}
This is not a coincidence. In fact, we show that for any $r\in\mathbb{R}\setminus\left\{1, \frac{1+\sqrt{13}}{2}, \frac{1-\sqrt{13}}{2}\right\}$, a linear rate $\frac{1}{r}$ is achieved for the above particular problem.
\begin{lemma}
    Let us consider the problem \eqref{main} with $\mathbb{H}=\mathbb{R}^n$, $A=0$ and $B=I$. Then, for $x_0\in\mathbb{H}$, $x_1 =\frac{x_0}{r}$, $x_2=\frac{x_0}{r^2}$ and $\alpha=0$,  \eqref{GFRB} satisfies $x_{k+1}=\frac{x_0}{r^k}$ for any real number $r$ except $r\in\left\{1, \frac{1+\sqrt{13}}{2}, \frac{1-\sqrt{13}}{2}\right\}$.
\end{lemma}
\begin{proof}
    With $A=0$, $B=I$, and fixed step-size $\lambda$, \eqref{GFRB} can be written as
    \begin{equation}\label{general roc 1}
        x_{k+1}=(1-\lambda(\delta+2))x_k+\lambda(2\delta+1)x_{k-1}-\lambda\delta x_{k-2}.
    \end{equation}
 For the particular choice $\lambda=\tfrac{1}{3(\delta+1)}$, \eqref{general roc 1} becomes
    \begin{equation}\label{general roc 2}
        x_{k+1}= \frac{2\delta+1}{3(\delta+1)}x_k+\frac{2\delta+1}{3(\delta+1)}x_{k-1}-\frac{\delta}{3(\delta+1)}x_{k-2}.
    \end{equation}
Since we want to show that $x_{k+1}=\frac{c}{r^k}$ for some $c>0$, for all $k\geq 1$, from \eqref{general roc 2}, we get
\begin{equation}\label{general roc 3}
\begin{split}
        \frac{1}{r^k}&=\frac{2\delta+1}{3(\delta+1)}\frac{1}{r^{k-1}}+\frac{2\delta+1}{3(\delta+1)}\frac{1}{r^{k-2}}-\frac{\delta}{3(\delta+1)}\frac{1}{r^{k-3}}\\ 
        \implies \frac{1}{r^3}&=\frac{2\delta+1}{3(\delta+1)}\frac{1}{r^2}+\frac{2\delta+1}{3(\delta+1)}\frac{1}{r}-\frac{\delta}{3(\delta+1)}.
\end{split}
\end{equation}
Simplifying \eqref{general roc 3}, we obtain $(3\delta+1)-(2\delta+1)r-(2\delta+1)r^2+\delta r^3=0$. This gives us the solution $\delta=\frac{r^2+r-3}{r^3-2r^2-2r+3}$. The denominator of $\delta$ factors as $(r-1)\bigl(r^2-r-3\bigr)$, so the excluded values are $r\in\{1,\tfrac{1-\sqrt{13}}{2},\tfrac{1+\sqrt{13}}{2}\}$. Hence, for any $r$ not in this set, choosing $\delta$ as above in~\eqref{general roc 2} gives a linear convergence rate $\frac{1}{r}$.
\end{proof}
\begin{figure}[htbp]
  \centering  

  \begin{subfigure}[b]{0.32\textwidth}
    \centering
    \includegraphics[width=\textwidth]{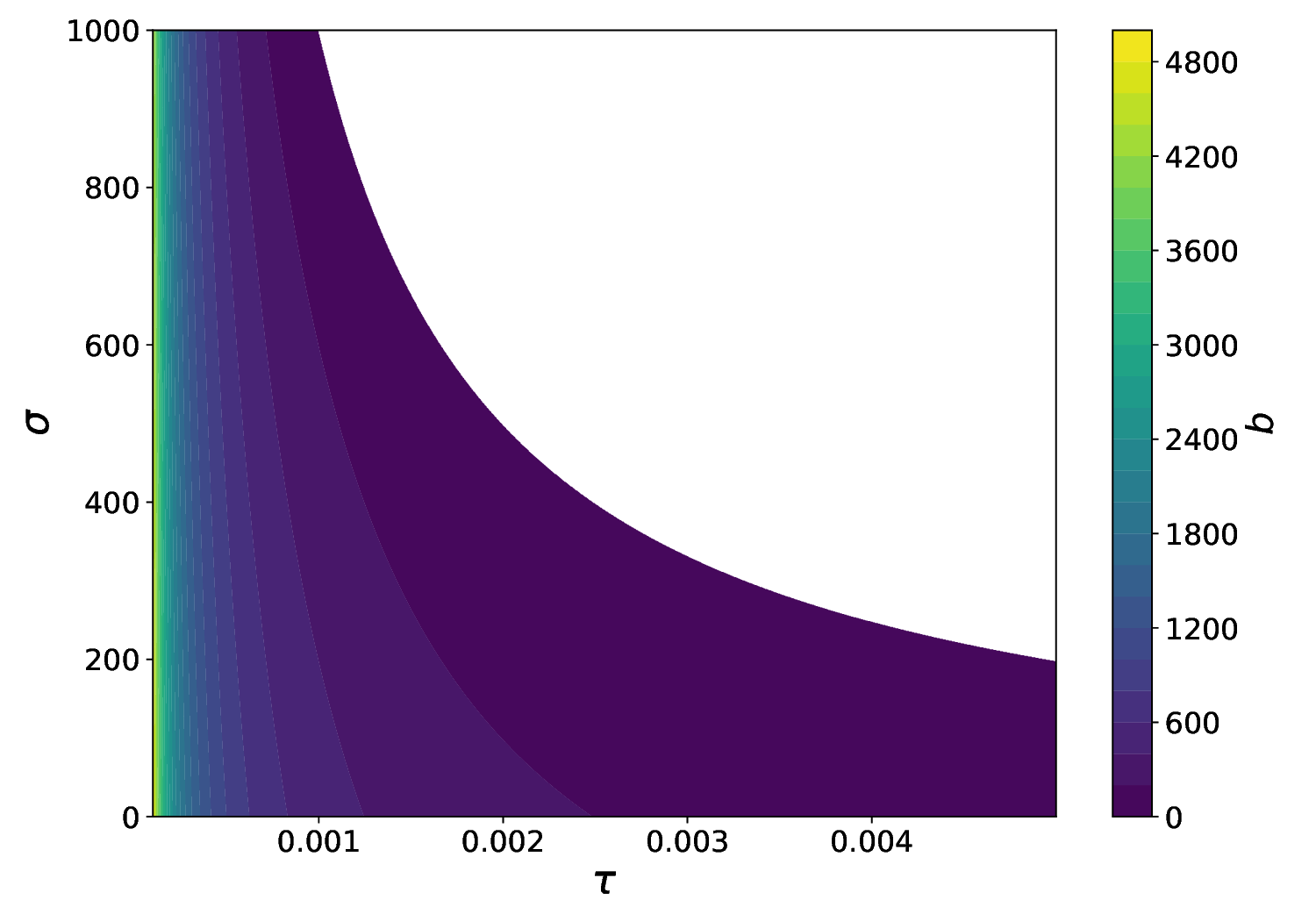}
    \caption{$\alpha=0.001$}
  \end{subfigure}\hfill
  \begin{subfigure}[b]{0.32\textwidth}
    \centering
    \includegraphics[width=\textwidth]{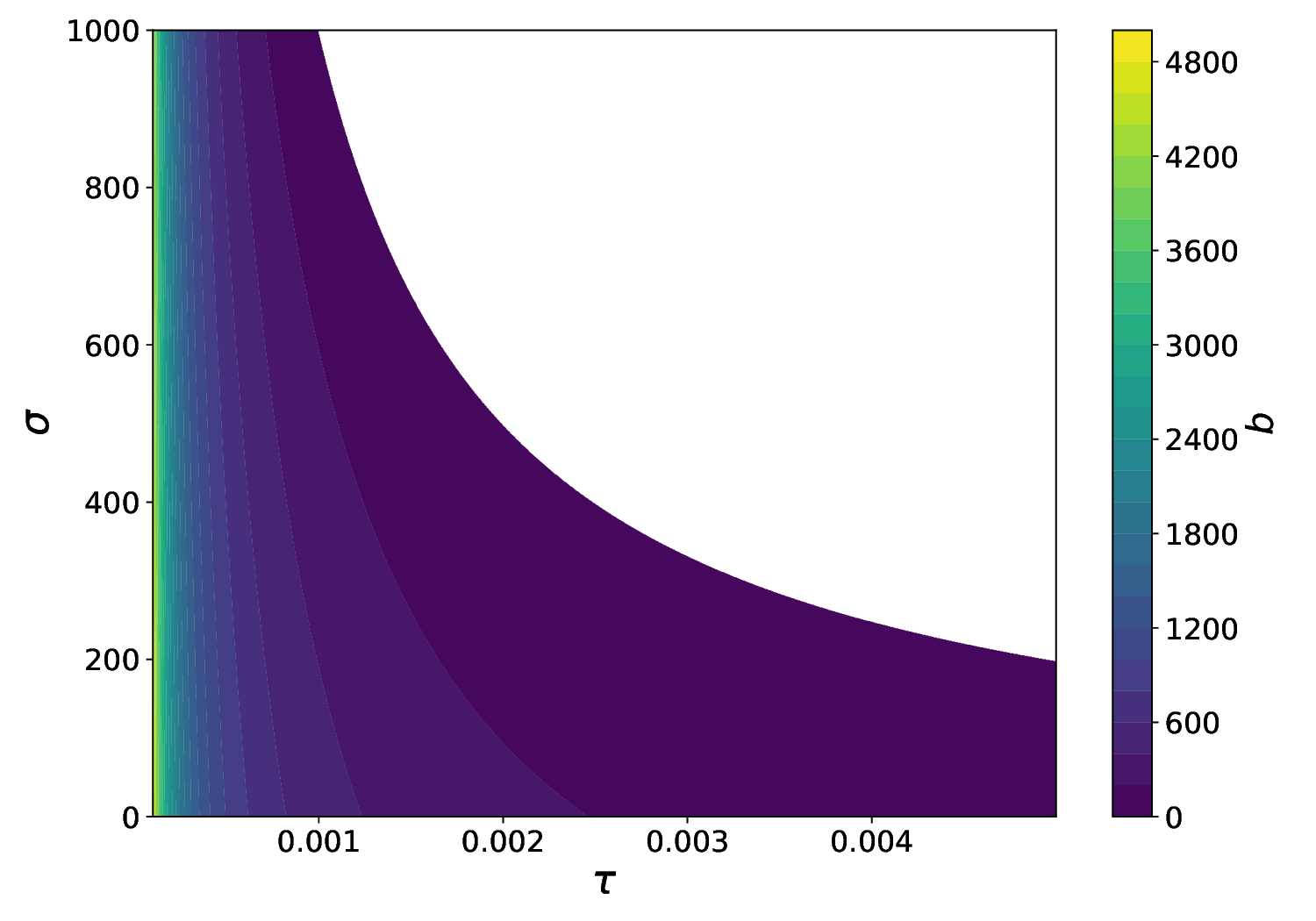}
    \caption{$\alpha=0.01$}
  \end{subfigure}\hfill
  \begin{subfigure}[b]{0.32\textwidth}
    \centering
    \includegraphics[width=\textwidth]{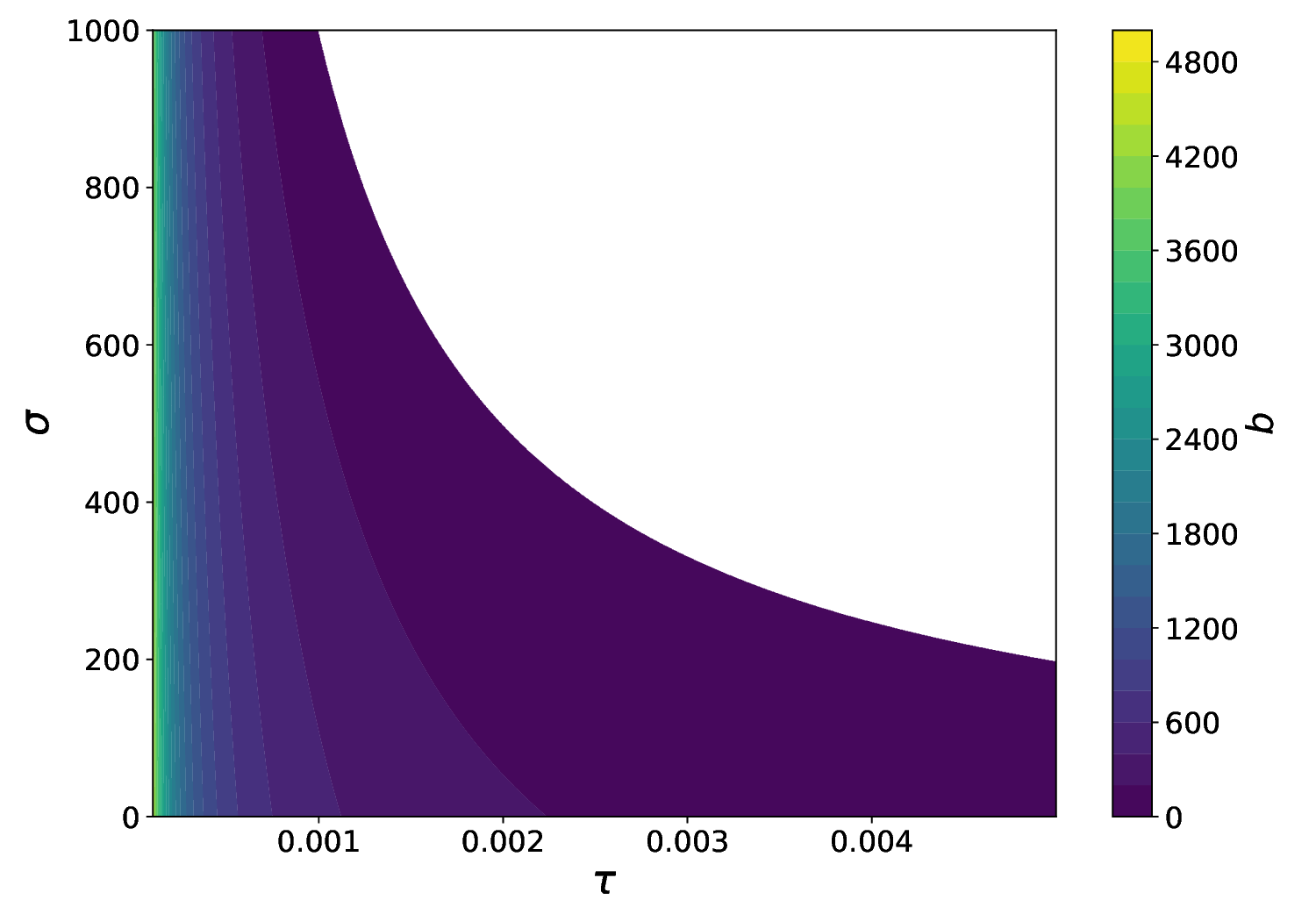}
    \caption{$\alpha=0.1$}
  \end{subfigure}

  \vspace{0.8em}

  \begin{subfigure}[b]{0.32\textwidth}
    \centering
    \includegraphics[width=\textwidth]{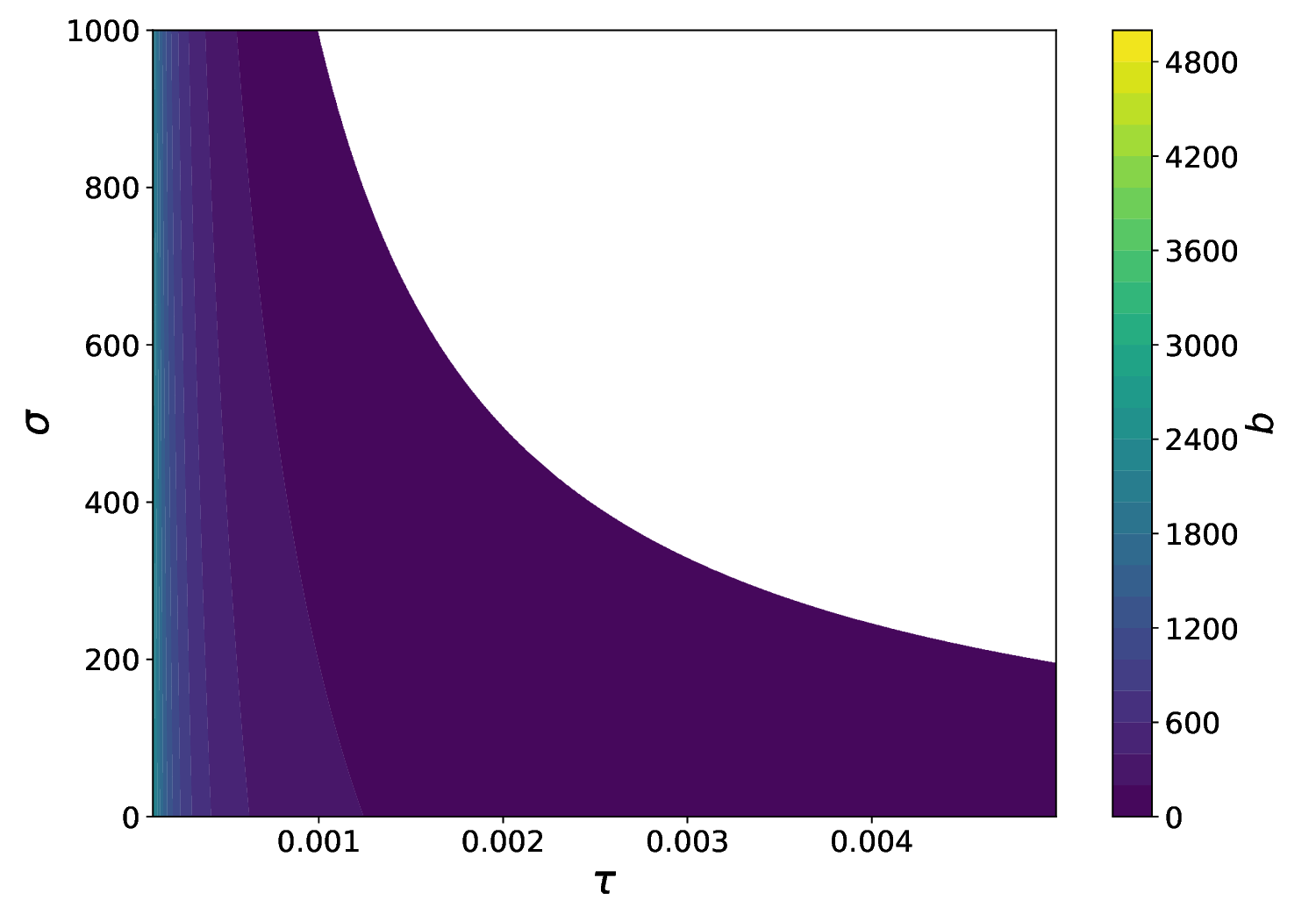}
    \caption{$\alpha=0.5$}
  \end{subfigure}\hfill
  \begin{subfigure}[b]{0.32\textwidth}
    \centering
    \includegraphics[width=\textwidth]{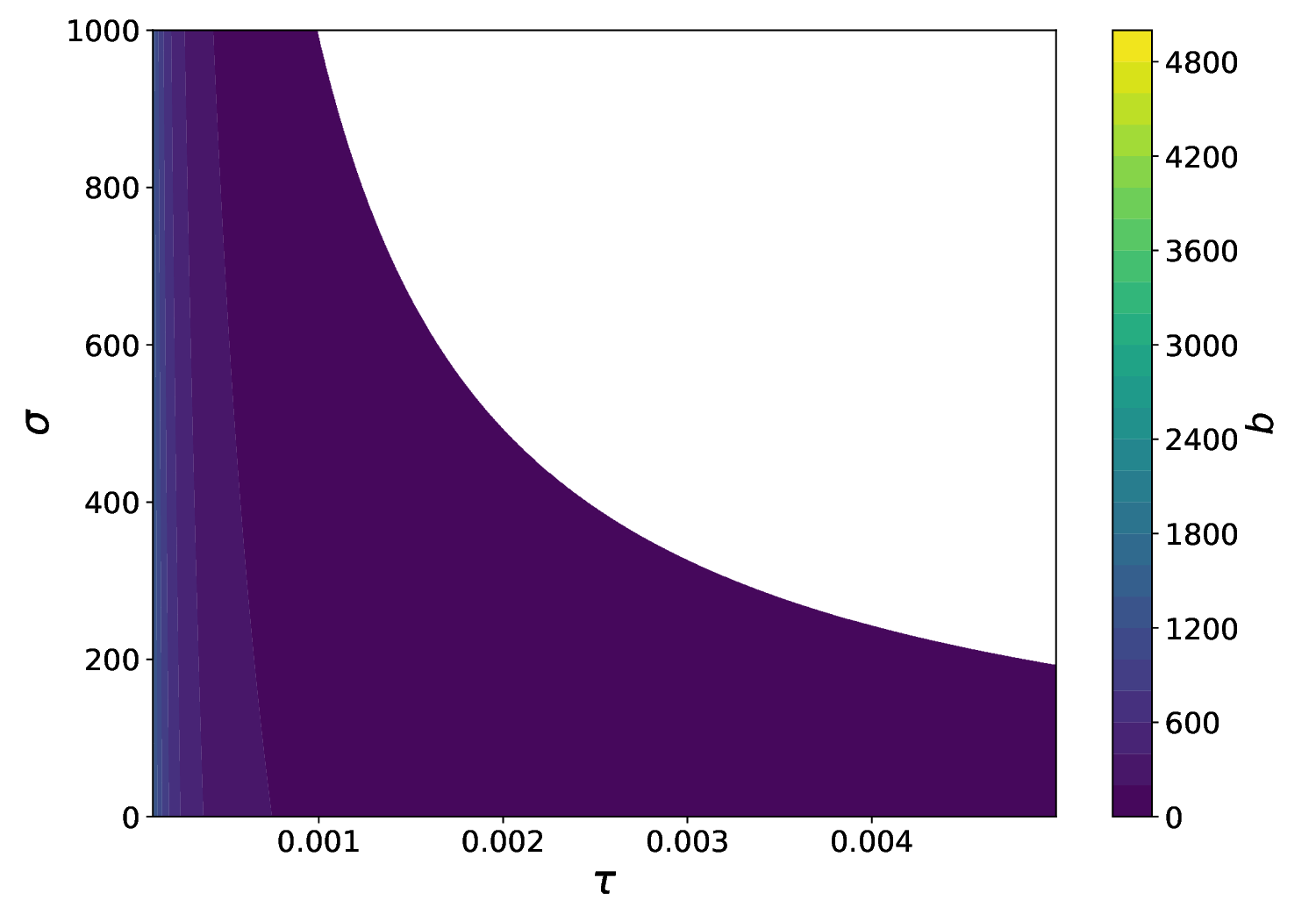}
    \caption{$\alpha=0.7$}
  \end{subfigure}\hfill
  \begin{subfigure}[b]{0.32\textwidth}
    \centering
    \includegraphics[width=\textwidth]{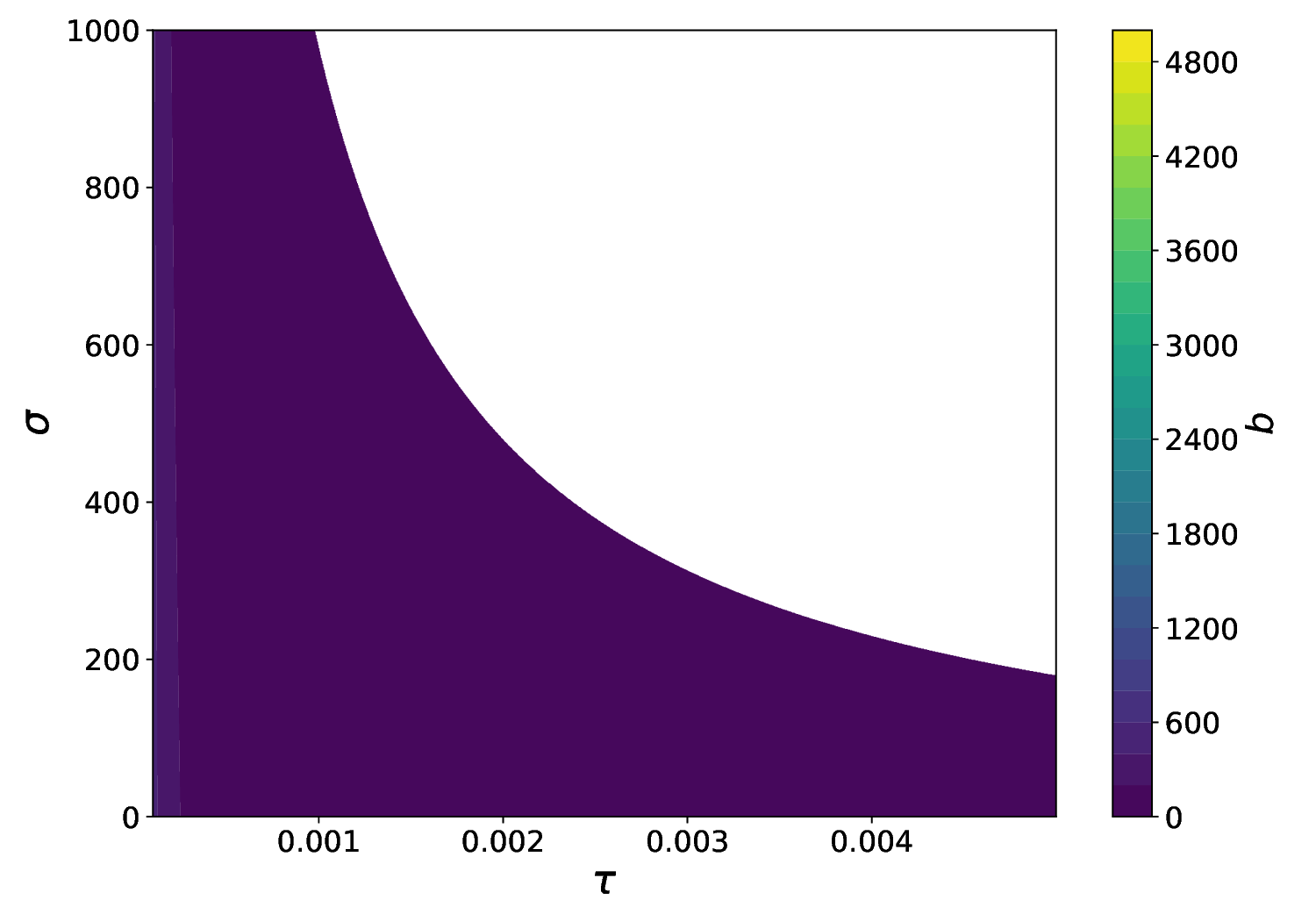}
    \caption{$\alpha=0.9$}
  \end{subfigure}

\caption{Admissible $(\tau,\sigma)$ regions for Algorithm~\ref{alg:EPDTR_splitting} across six values of $\alpha$, obtained by plotting $b(\tau,\sigma;\alpha)$.}
  \label{fig:admissible_regions}
\end{figure}

\section{An Extended Primal-Dual Twice Reflected Algorithm}
Let $\mathbb{H}_1$ and $\mathbb{H}_2$ be finite-dimensional real Hilbert spaces. We consider the following primal monotone inclusion problem: find $x \in \mathbb{H}_1$ such that
\begin{equation}\label{primal three operator inlcusion}
    0\in (A+B+K^*CK)(x),
\end{equation}
where $A:\mathbb{H}_1\rightrightarrows \mathbb{H}_1$ and $C:\mathbb{H}_2\rightrightarrows \mathbb{H}_2$ are maximally monotone operators, $B:\mathbb{H}_1\rightrightarrows \mathbb{H}_1$ is monotone and $L$-Lipschitz, and $K:\mathbb{H}_1\rightrightarrows \mathbb{H}_2$ is a linear operator with adjoint $K^*$. Problems of the form \eqref{primal three operator inlcusion} have been studied in the literature; see, for example \cite{condat2013primal, combettes2012primal, rieger2020backward}. Recently, Soe et al. \cite{soe2025golden} addressed the setting in which the Lipschitz constant of $B$ is either unknown or only locally available. In this section, we propose an algorithm to solve \eqref{primal three operator inlcusion} which can be seen as an extension of the \emph{primal-dual twice reflected algorithm} (PDTR) proposed in \cite[Section 3]{malitsky2023first}. The associated dual monotone inclusion, as considered in \cite{attouch1996general, malitsky2023first}, is to find $y \in \mathbb{H}_2$ such that
\begin{equation}\label{dual three operator inlcusion}
    0\in -K(A+B)^{-1}(K^*y)+ C^{-1}(y).
\end{equation}
Notice that both the primal \eqref{primal three operator inlcusion} and dual \eqref{dual three operator inlcusion} inclusion problem can be expressed as 
\begin{equation}\label{primal-dual system expanded}
\begin{cases}
0 \in A x + Bx + K^* y, \\
0 \in -Kx + C^{-1} y,
\end{cases}
\quad (x,y) \in \mathbb{H}_1 \times \mathbb{H}_2.
\end{equation}
Equivalently, this can be written as the following primal-dual system,
\begin{equation}\label{primal-dual system compact}
0 \in 
\begin{bmatrix}
A & K^* \\
- K & C^{-1}
\end{bmatrix}
\begin{bmatrix}
x \\ y
\end{bmatrix}
+
\begin{bmatrix}
B & 0 \\
0 & 0
\end{bmatrix}
\begin{bmatrix}
x \\ y
\end{bmatrix}.
\end{equation}

We now have the following primal--dual algorithm to solve both the primal \eqref{primal three operator inlcusion} and dual~\eqref{dual three operator inlcusion} inclusion problems simultaneously.
\begin{algorithm}[H]
\caption{Extended Primal--Dual Twice‐Reflected Algorithm (EPDTR) for \eqref{primal three operator inlcusion}}
\label{alg:EPDTR_splitting}
\begin{algorithmic}[1]\setlength{\itemsep}{0pt}\setlength{\parskip}{0pt}
\State \textbf{Initialize:} Let $x_{-2},x_{-1},x_{0}\in\mathbb{H}_1$, $y_{-1},y_{0}\in\mathbb{H}_2$,
$\alpha\in[0,1)$, and $b\in\mathbb{R}$. Choose $\tau,\sigma>0$ such that 
$2\tau(1+|b|)L + (1-\alpha)\tau\sigma\|K\|^2<1-\alpha.$
\For{$k=0,1,2,\dots$}
  \Statex
  \begin{subequations}\label{eq:EPDTR_updates}
    \begin{align}
      x_{k+1}
      &=J_{\tau A}\Bigl((1-\alpha)x_k+\alpha x_{k-1}-\tau K^*y_k-(b+2)\tau Bx_k
        +(2b+1)\tau Bx_{k-1}-b\tau Bx_{k-2}\Bigr),\nonumber\\
      y_{k+1}
      &=J_{\sigma C^{-1}}\Bigl((1-\alpha)y_k+\alpha y_{k-1}+2\sigma Kx_{k+1}
      -\sigma K\bigl((1-\alpha)x_k+\alpha x_{k-1}\bigr)\Bigr).
      \label{eq:EPDTR_y}
    \end{align}
  \end{subequations}
\EndFor
\end{algorithmic}
\end{algorithm}

Prior to discussing the convergence of Algorithm~\ref{alg:EPDTR_splitting}, some comments are in order.

\begin{remark}
Plugging $b=\alpha=0$ into Algorithm~\ref{alg:EPDTR_splitting} yields the PDTR scheme of Malitsky and Tam~\cite{malitsky2023first}:
\[
  \begin{cases}
    x_{k+1}
     =J_{\tau A}\bigl(x_k-\tau K^*y_k-2\tau Bx_k+\tau Bx_{k-1}\bigr),\\
    y_{k+1}
     =J_{\sigma C^{-1}}\bigl(y_k+2\sigma Kx_{k+1}-\sigma Kx_k\bigr).
  \end{cases}
\]
Moreover, the step-size condition of Algorithm~\ref{alg:EPDTR_splitting} reduces to $2\tau L+\tau\sigma\|K\|^2<1$, which is the step-size condition for PDTR.
\end{remark}

\begin{remark}
From the step-size condition of Algorithm \ref{alg:EPDTR_splitting}, one can notice that
\[
2\tau L + (1-\alpha)\tau\sigma\|K\|^2 < 1-\alpha-2\tau|b|L.
\]
Thus, for fixed $L$ and $\|K\|$, increasing $b$ tightens the admissible region for $(\tau,\sigma)$. This shrinkage is illustrated in Figure \ref{fig:admissible_regions} with $L = K = 1$ as an example.
\end{remark}
\subsection{Derivation as a GFRB and Convergence}
Before proving convergence, recall that PDTR can be expressed as a forward--reflected--backward method with suitable metrics \cite{malitsky2023first}. We now show that Algorithm~\ref{alg:EPDTR_splitting} can be written as an \emph{generalized forward--reflected--backward} (GFRB) scheme~\cite{dung2025generalization}. To do so, we consider the following operators, which were considered in \cite{malitsky2023first} to analyze PDTR.

Set $\mathbb{H}:=\mathbb{H}_1\times\mathbb{H}_2$ and $z_k:=(x_k,y_k)$. Define
\[
G:=
\begin{bmatrix}
A & K^*\\
- K & C^{-1}
\end{bmatrix},
\qquad
F:=
\begin{bmatrix}
B & 0\\
0 & 0
\end{bmatrix},
\qquad
M:=
\begin{bmatrix}
\frac{1}{\tau}I & -K^*\\
- K & \frac{1}{\sigma}I
\end{bmatrix},
\]
which is self-adjoint and positive definite whenever $\tau\sigma\|K\|^2<1$. Now, by applying the first-order optimality conditions to \eqref{eq:EPDTR_updates}, we obtain
\begin{align}
\frac{1}{\tau}&\bigl((1-\alpha)x_k+\alpha x_{k-1}\bigr)-K^*y_k
-(b+2)Bx_k+(2b+1)Bx_{k-1}-bBx_{k-2}
\in \left(A+\frac{1}{\tau}I\right)(x_{k+1}),
\nonumber\\&
\frac{1}{\sigma}\bigl((1-\alpha)y_k+\alpha y_{k-1}\bigr)-K\bigl((1-\alpha)x_k+\alpha x_{k-1}\bigr)
\in -2Kx_{k+1}+\left(C^{-1}+\frac{1}{\sigma}I\right)y_{k+1}.
\label{eq:EPDTR_optimality}
\end{align}
Let $\bar z_k:=(1-\alpha)z_k+\alpha z_{k-1}$. Then we can rewrite \eqref{eq:EPDTR_optimality} equivalently as
\begin{equation*}
M\bar z_k-(b+2)F(z_k)+(2b+1)F(z_{k-1})-bF(z_{k-2})
\in (M+G)(z_{k+1}).
\end{equation*}
Therefore, by the definition of the resolvent,
\begin{equation}\label{eq:EPDTR_as_inertial_GFRB}
z_{k+1}
=
J_{M^{-1}G}\Bigl(
\bar z_k
-(b+2)\,M^{-1}F(z_k)
+(2b+1)\,M^{-1}F(z_{k-1})
-b\,M^{-1}F(z_{k-2})
\Bigr),
\end{equation}
which is precisely a GFRB iteration on $\mathbb{H}$. To conclude the proof of the convergence of Algorithm \ref{alg:EPDTR_splitting}, we prove the following result.

\begin{theorem}\label{thm:EPDTR_convergence}
Let $\{(x_k,y_k)\}$ be the sequence generated by Algorithm~\ref{alg:EPDTR_splitting}.  
Suppose that the solution set of \eqref{primal three operator inlcusion} is nonempty and the parameters $\tau,\sigma>0$ satisfy
$2\tau(1+|b|)L + (1-\alpha)\tau\sigma\|K\|^2 < 1-\alpha.$
Then $\{(x_k, y_k)\}$ converges to a solution of \eqref{primal-dual system compact}.
\end{theorem}

\begin{proof}
We need only to show that \eqref{eq:EPDTR_as_inertial_GFRB} is an instance of the GFRB method of \cite{dung2025generalization}
and satisfies all its assumptions. Set $\mathbb{H}:=\mathbb{H}_1\times\mathbb{H}_2$ and define
\[
\tilde G:=M^{-1}G,
\qquad
\tilde F:=M^{-1}F.
\]
By \cite[Lemma~3.2]{malitsky2023first}, $\tilde G$ is maximally monotone on $\mathbb{H}$, and $\tilde F$ is monotone and $L_M$-Lipschitz with
\[
L_M=\frac{\tau L}{1-\tau\sigma\|K\|^2}.
\]
Moreover, \eqref{eq:EPDTR_as_inertial_GFRB} is exactly the  GFRB iteration of \cite{dung2025generalization} on $\mathbb{H}$
with unit step-size applied to $0\in(\tilde G+\tilde F)(z)$. Finally, the condition $2\tau(1+|b|)L+(1-\alpha)\tau\sigma\|K\|^2<1-\alpha$ is equivalent to
\[
\alpha+\frac{2\tau(1+|b|)L}{1-\tau\sigma\|K\|^2}<1,
\quad\text{i.e.,}\quad
1<\frac{1-\alpha}{2(|b|+1)L_M},
\]
which is precisely the step-size condition required for GFRB in \cite[Theorem~3.1]{dung2025generalization}. Therefore, by \cite[Theorem~3.1]{dung2025generalization},
$(z_k)$ converges to some $\bar z\in(\tilde G+\tilde F)^{-1}(0)$. Since $M$ is invertible, $(\tilde G+\tilde F)^{-1}(0)=(G+F)^{-1}(0)$, and hence
$\bar z=(\bar x,\bar y)$ solves \eqref{primal-dual system compact}. Thus, $(x_k,y_k)\to(\bar x,\bar y)$, and the proof is complete.
\end{proof}


\section{Numerical Experiments}
In this section, we present numerical experiments to evaluate the performance of Algorithm~\ref{alg:forb_incr}, and compare it with the following methods: forward-reflected-backward (FRB) with linesearch \cite[Algorithm~1]{malitsky2020forward}, modified FRB \cite[Algorithm~3.1]{hieu2021modified}, and the perturbed reflected-forward-backward (RFB) method \cite[Algorithm~1]{tan2025perturbed}. For each algorithm, we plot the error, $\mathrm{err}_k = \|x_{k+1} - x_k\|$ versus the number of iterations and report the CPU time (in seconds) required to reach specified error thresholds.

All methods were implemented in Python~3.11 and executed in a Google Colab environment equipped with an Intel Xeon CPU @ 2.20~GHz and 12.7~GB RAM.

\begin{example}[Convex minimization problem]\label{example_1}
We consider the following convex minimization problem:
\begin{equation*}
        \min_{x \in \mathbb{R}^m} \; \frac{1}{2}\|x\|_2^2 + b^{\top}x + 3 + \|x\|_1,
\end{equation*}
where $b\in\mathbb{R}^m$. The above problem can be equivalently written as the monotone inclusion problem~\eqref{main} with $A = \partial \|x\|_1$ and $B = 2x+b$. 

\begin{figure}[htbp]
  \centering

  \begin{subfigure}[b]{0.32\textwidth}
    \centering
    \includegraphics[width=\linewidth]{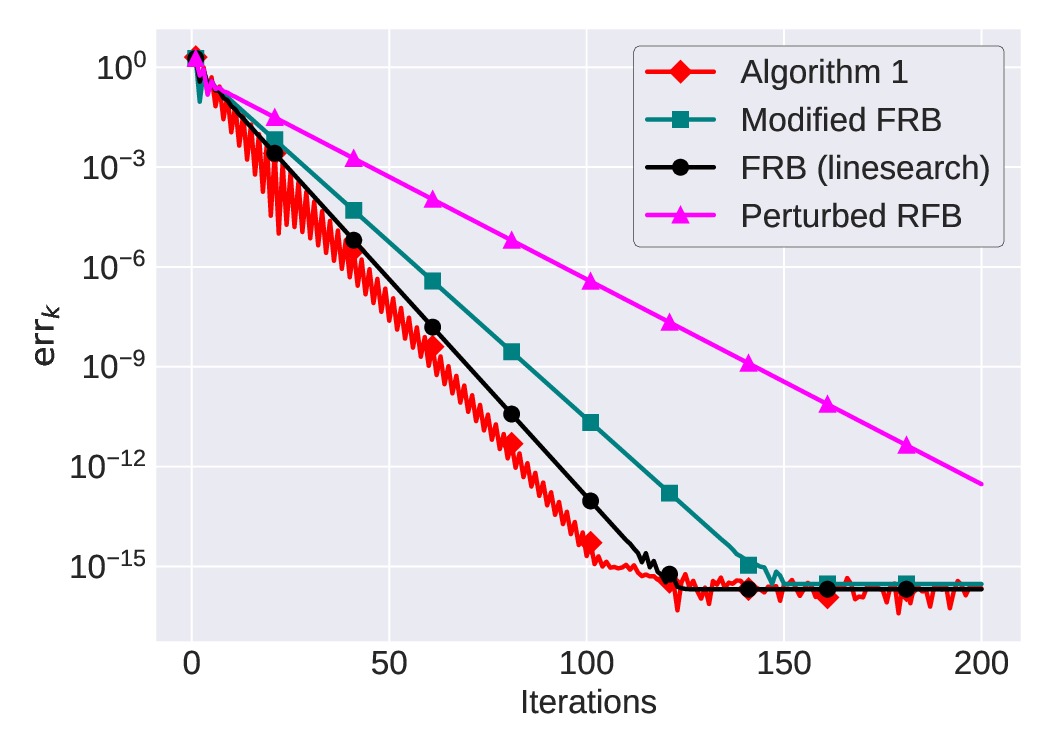}
    \caption{$m=500$}
    \label{fig:panel_a}
  \end{subfigure}\hfill
  \begin{subfigure}[b]{0.32\textwidth}
    \centering
    \includegraphics[width=\linewidth]{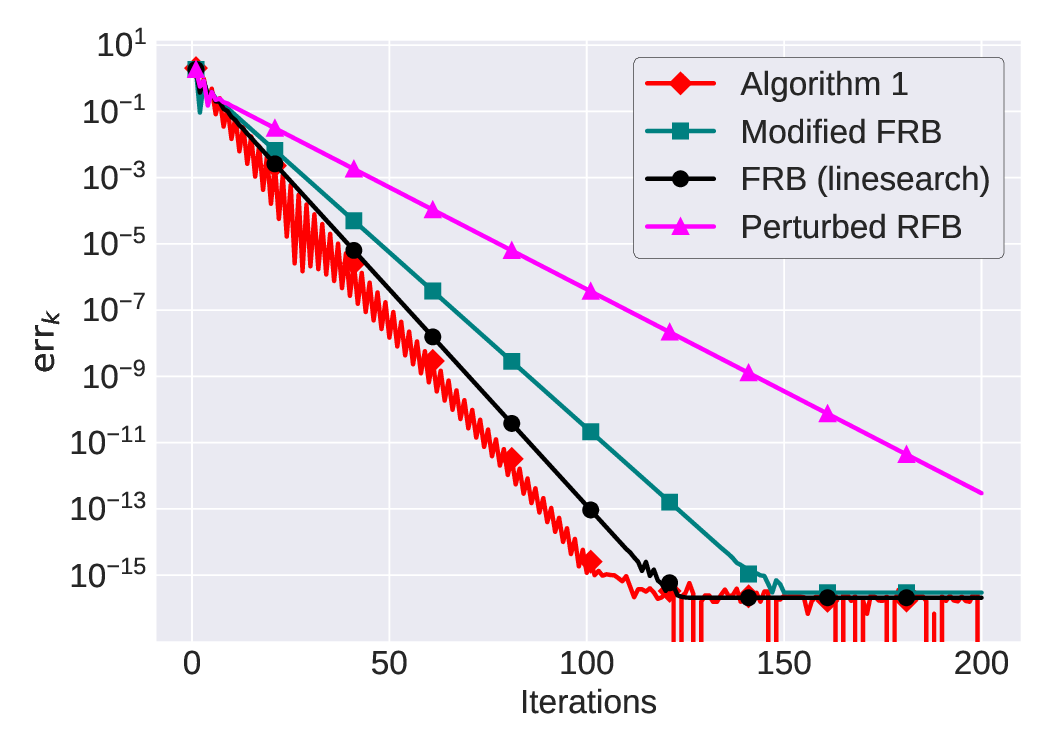}
    \caption{$m=500$}
    \label{fig:panel_b}
  \end{subfigure}\hfill
  \begin{subfigure}[b]{0.32\textwidth}
    \centering
    \includegraphics[width=\linewidth]{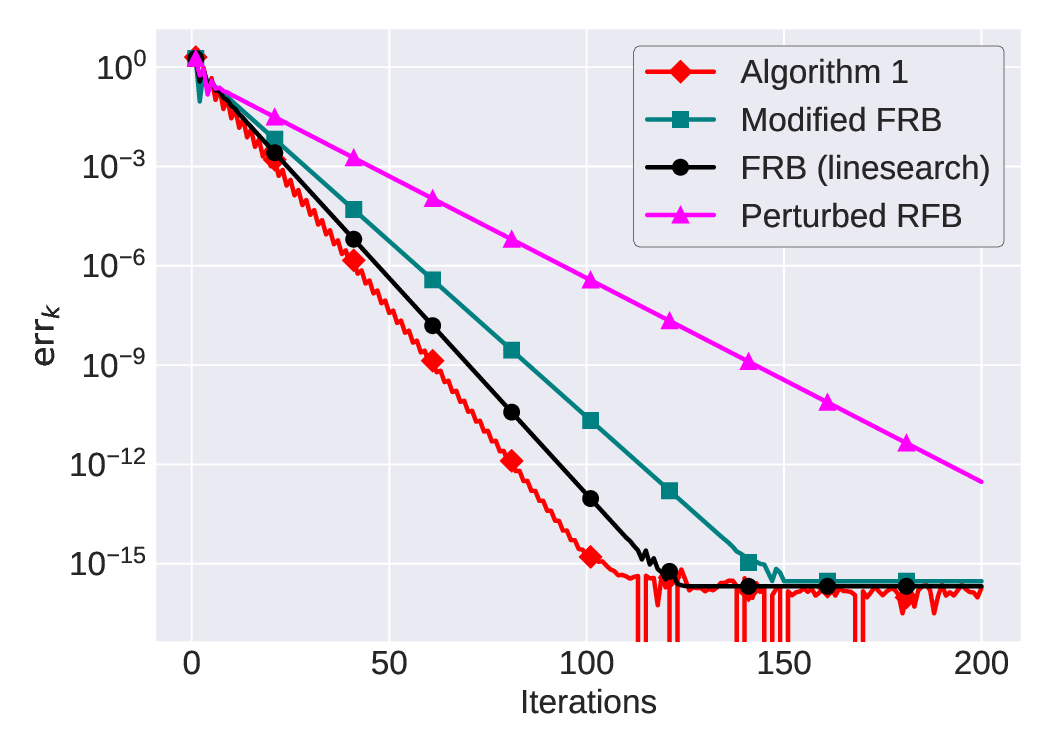}
    \caption{$m=500$}
    \label{fig:panel_c}
  \end{subfigure}

  \vspace{1em}

  \begin{subfigure}[b]{0.32\textwidth}
    \centering
    \includegraphics[width=\linewidth]{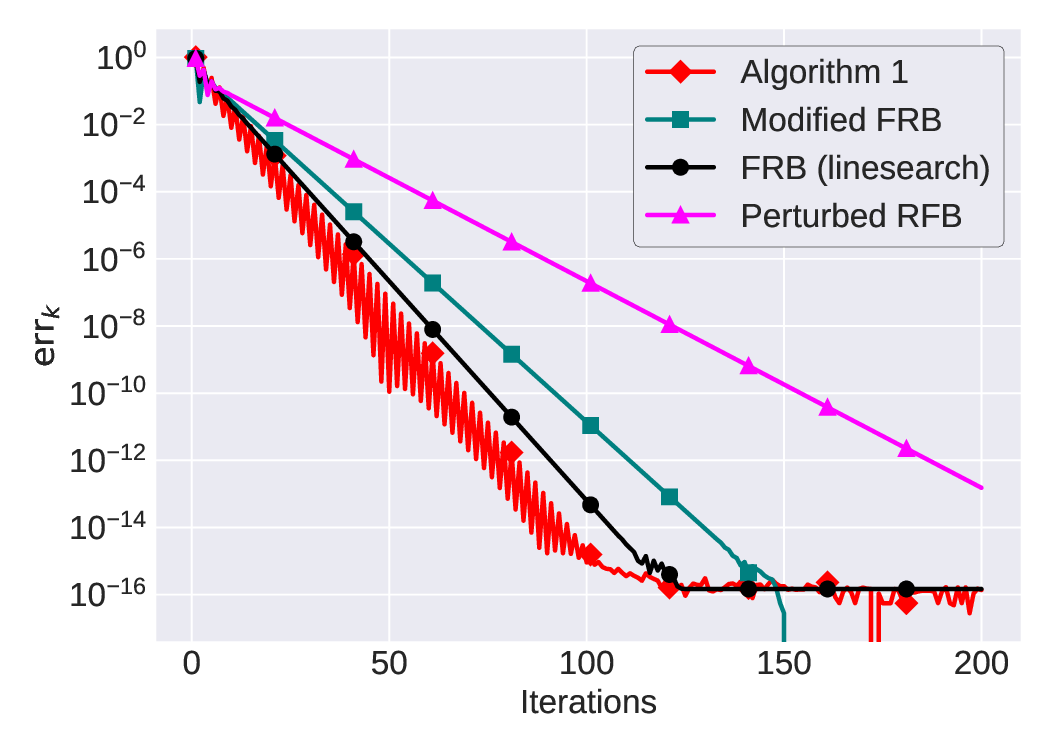}
    \caption{$m=200$}
    \label{fig:panel_d}
  \end{subfigure}\hfill
  \begin{subfigure}[b]{0.32\textwidth}
    \centering
    \includegraphics[width=\linewidth]{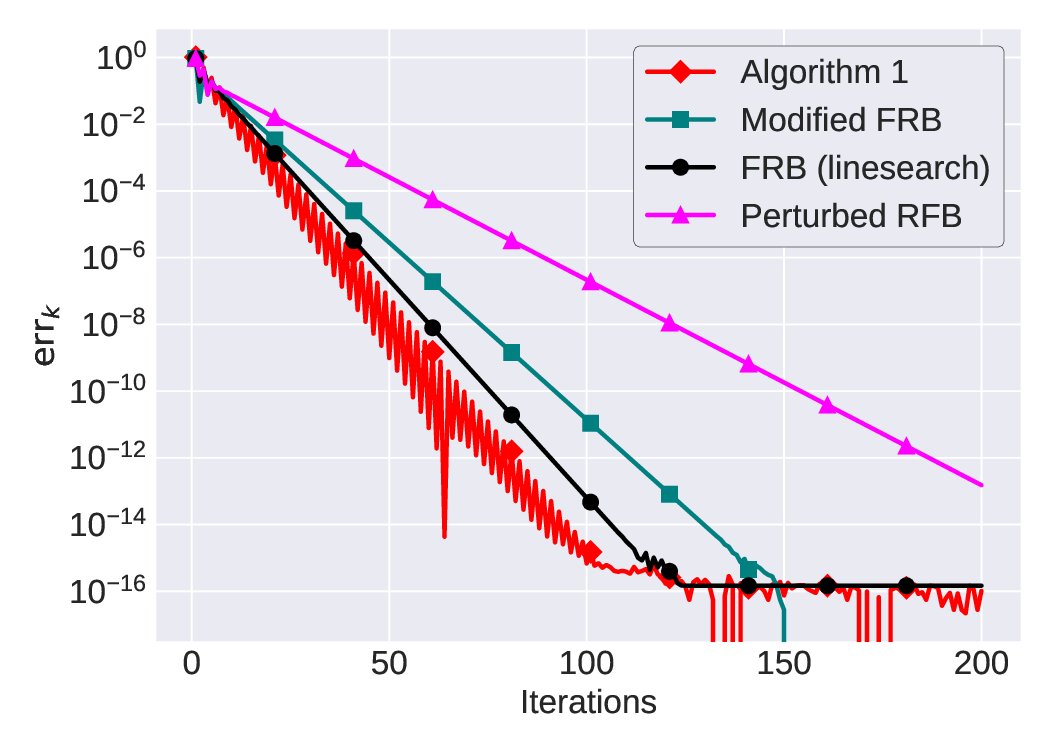}
    \caption{$m=200$}
    \label{fig:panel_e}
  \end{subfigure}\hfill
  \begin{subfigure}[b]{0.32\textwidth}
    \centering
    \includegraphics[width=\linewidth]{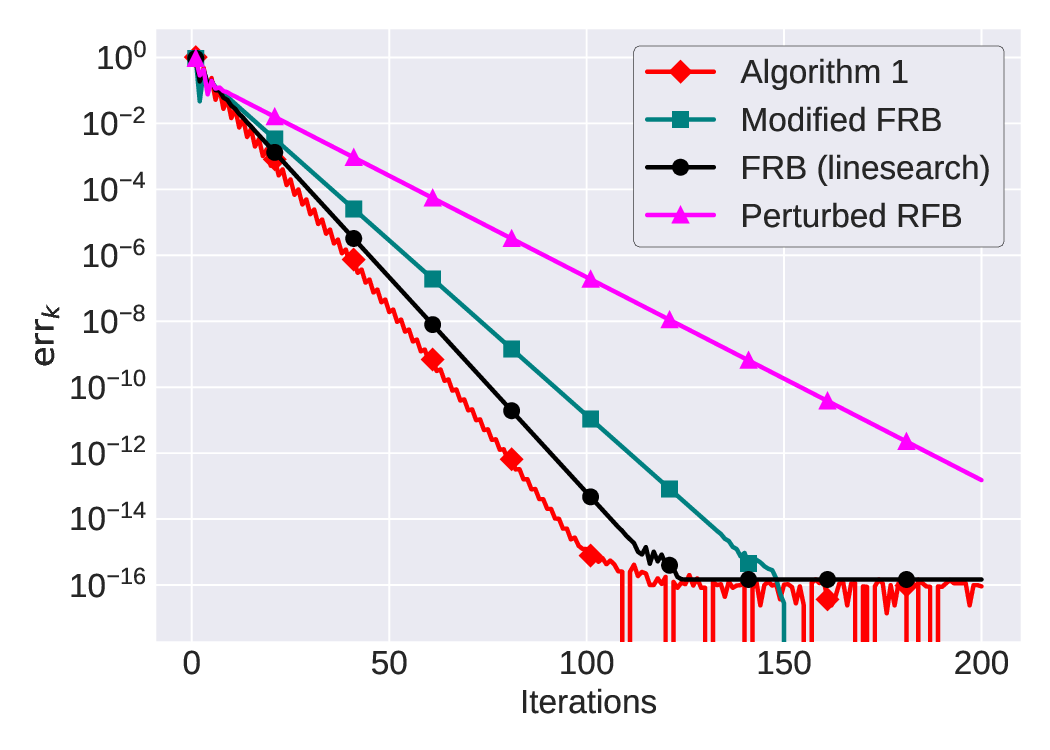}
    \caption{$m=200$}
    \label{fig:panel_f}
  \end{subfigure}

\caption{Comparison of $\mathrm{err}_k$ for Example~\ref{example_1} for $m\in\{200,500\}$. 
The parameter choices in Algorithm~\ref{alg:forb_incr} are selected as: for (a) and (d), \((\alpha,\delta)=(10^{-2},10^{-2})\); for (b) and (e), \((\alpha,\delta)=(0,10^{-2})\); for (c) and (f) \((\alpha,\delta)=(0,0)\).}
\label{fig:delta_and_comparison}
\end{figure}
Then the resolvent of $A $ is given by
$$
J_{\lambda A}(z) = \mathrm{prox}_{\lambda \|\cdot\|_1}(z) = \operatorname{soft}_{\lambda}(z),
$$
where the \textit{soft-thresholding operator} is defined componentwise as
$$
\left[\operatorname{soft}_{\lambda}(z)\right]_i = \begin{cases}
z_i - \lambda, & \text{if } z_i > \lambda, \\
0, & \text{if } |z_i| \leq \lambda, \\
z_i + \lambda, & \text{if } z_i < -\lambda.
\end{cases}
$$

\begin{table}[htbp]
\centering
\caption{Iteration counts and CPU time (in seconds) to reach tolerance \(10^{-7}\) for Example~\ref{example_1} with different dimensions \(m\). For Algorithm~1, we use \(\alpha=10^{-3},\delta=10^{-2}\).}
\label{tab:l1_benchmark_combined}
\begin{tabular}{lcccccc}
\toprule
 & \multicolumn{2}{c}{\(m=200\)} & \multicolumn{2}{c}{\(m=500\)} & \multicolumn{2}{c}{\(m=1000\)} \\
\cmidrule(lr){2-3}\cmidrule(lr){4-5}\cmidrule(lr){6-7}
Algorithm & Iter. & CPU (s) & Iter. & CPU (s) & Iter. & CPU (s) \\
\midrule
Algorithm~\ref{alg:forb_incr} & 42  & 0.002771 & 44  & 0.002949 & 46  & 0.003205 \\
Modified FRB                  & 64  & 0.002899 & 67  & 0.003330 & 68  & 0.002289 \\
FRB (linesearch)              & 53  & 0.004873 & 55  & 0.005342 & 56  & 0.003574 \\
Perturbed RFB                 & 100 & 0.006444 & 100 & 0.006678 & 100 & 0.004413 \\
\bottomrule
\end{tabular}
\end{table}

\begin{table}[htbp]
\centering
\caption{Iteration counts and CPU time (in seconds) to reach tolerance \(10^{-7}\) for Example~\ref{example_1} with different dimensions \(m\). For Algorithm~\ref{alg:forb_incr}, we use \(\alpha=0\) and \(\delta=10^{-2}\).}
\label{tab:l1_benchmark_delta1e-2_alpha0}
\begin{tabular}{lcccccc}
\toprule
 & \multicolumn{2}{c}{\(m=1500\)} & \multicolumn{2}{c}{\(m=2000\)} & \multicolumn{2}{c}{\(m=3000\)} \\
\cmidrule(lr){2-3}\cmidrule(lr){4-5}\cmidrule(lr){6-7}
Algorithm & Iter. & CPU (s) & Iter. & CPU (s) & Iter. & CPU (s) \\
\midrule
Algorithm~\ref{alg:forb_incr} & 46  & 0.004490 & 46  & 0.004939 & 48  & 0.006482 \\
Modified FRB                  & 69  & 0.004869 & 70  & 0.005288 & 70  & 0.006445 \\
FRB (linesearch)              & 57  & 0.007959 & 57  & 0.008465 & 58  & 0.010938 \\
Perturbed RFB                 & 114 & 0.010630 & 115 & 0.011794 & 116 & 0.014681 \\
\bottomrule
\end{tabular}
\end{table}

In this experiment, we initialize all the algorithms with $x_0=x_1=0$. Here, the elements of $b$ are drawn from $\mathcal{N}(0,1)$ and the others parameters are selected as follows. 

\begin{itemize}
    \item \textbf{Algorithm \ref{alg:forb_incr}}: $\varepsilon = 10^{-12}$, $\lambda_0=\lambda_{-1} =0.2$, $c_2= \frac{0.9(1-\varepsilon-\alpha)}{2|\delta|+2}$ for given $\alpha,\delta$, and $c_1=0.9c_2$, $\gamma_k =\frac{0.1}{k^{1.001}}$.\
    \item \textbf{Modified FRB} \cite{hieu2021modified}:  $\lambda_0=\lambda_{-1} =0.2$ and $\mu=0.4$ as defined in the algorithm.
    \item \textbf{FRB (linesearch)} \cite{malitsky2020forward}: $\lambda_0=\lambda_{-1} =0.2$ and $\delta=0.9, \sigma=0.7$ used inside the loop of the linesearch.
    \item \textbf{Perturbed RFB} \cite{tan2025perturbed}:  $\lambda_0=\lambda_{-1} =0.2$, $\mu=0.19$ and $\gamma_k =\frac{0.1}{k^{1.001}}$.
\end{itemize}

Figure~\ref{fig:delta_and_comparison} presents the results of this experiment, showing that Algorithm~\ref{alg:forb_incr} achieves a faster decrease of $\mathrm{err}_k$ with respect to the number of iterations, whereas Table~\ref{tab:l1_benchmark_combined} and \ref{tab:l1_benchmark_delta1e-2_alpha0} collects the CPU times required to reach a given error bound.
\end{example}

\begin{example}[Affine monotone inclusion]\label{example_2}
Fix $m\in\mathbb{N}$. Let $E\in\mathbb{R}^{m\times m}$ be symmetric and define
\[
A(x) := (E+\beta I)x,\qquad x\in\mathbb{R}^m,
\]
where $\beta\in\mathbb{R}$ is chosen so that $E+\beta I$ is positive semidefinite. Set
\[
B(x) := Mx + b,\qquad M := G^{\top} + S + D,\qquad b\in\mathbb{R}^m,
\]
with $G\in\mathbb{R}^{m\times m}$ chosen randomly, $S$ skew–symmetric, and $D=\operatorname{diag}(d_1,\dots,d_m)$ is positive- semi definite. 

Note that if we take $\beta \;=\; \max_{1\le i\le m} |\lambda_i(E)|$, then $E+\beta I$ is positive semidefinite, and hence $A$ is monotone. Moreover, $I+A = (1+\beta)I + E$  has eigenvalues $1+\beta+\lambda_i(E)\ge 1$, hence $I+A$ is onto. Therefore, by Minty’s theorem \cite[Theorem 21.1]{bauschke2011convex}, $A$ is maximally monotone.
\begin{figure}[htbp]
  \centering

  \begin{subfigure}[b]{0.32\textwidth}
    \centering
    \includegraphics[width=\linewidth]{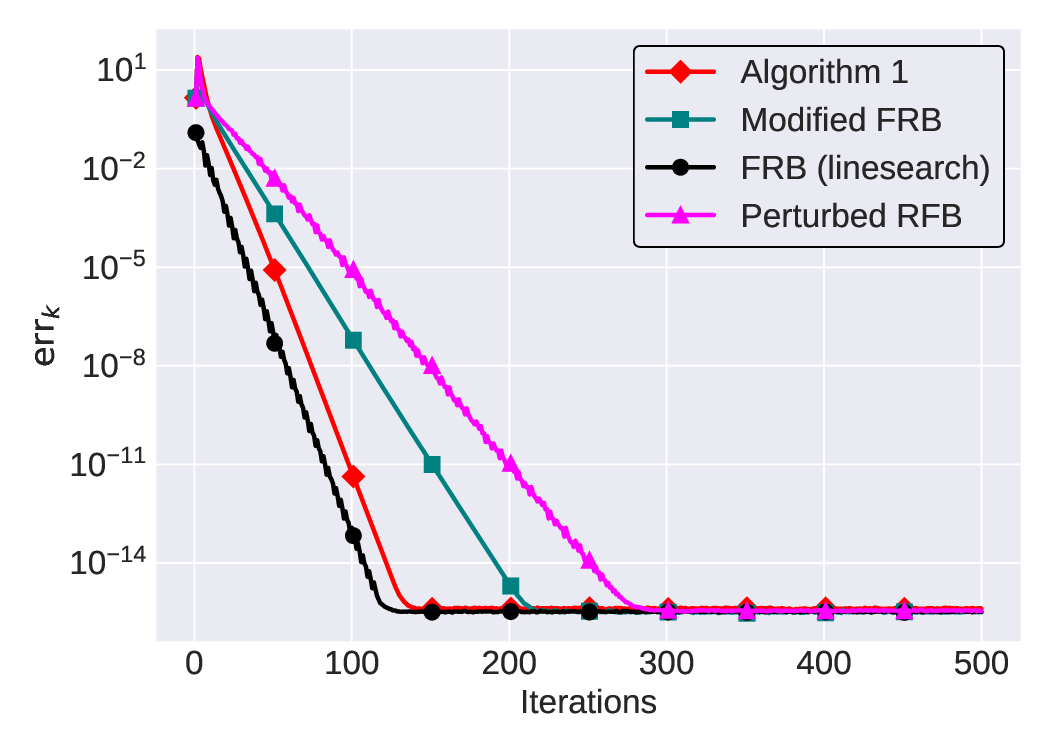}
    \caption{$m=1000$}
    \label{fig:ex2_delta_m300}
  \end{subfigure}\hfill
  \begin{subfigure}[b]{0.32\textwidth}
    \centering
    \includegraphics[width=\linewidth]{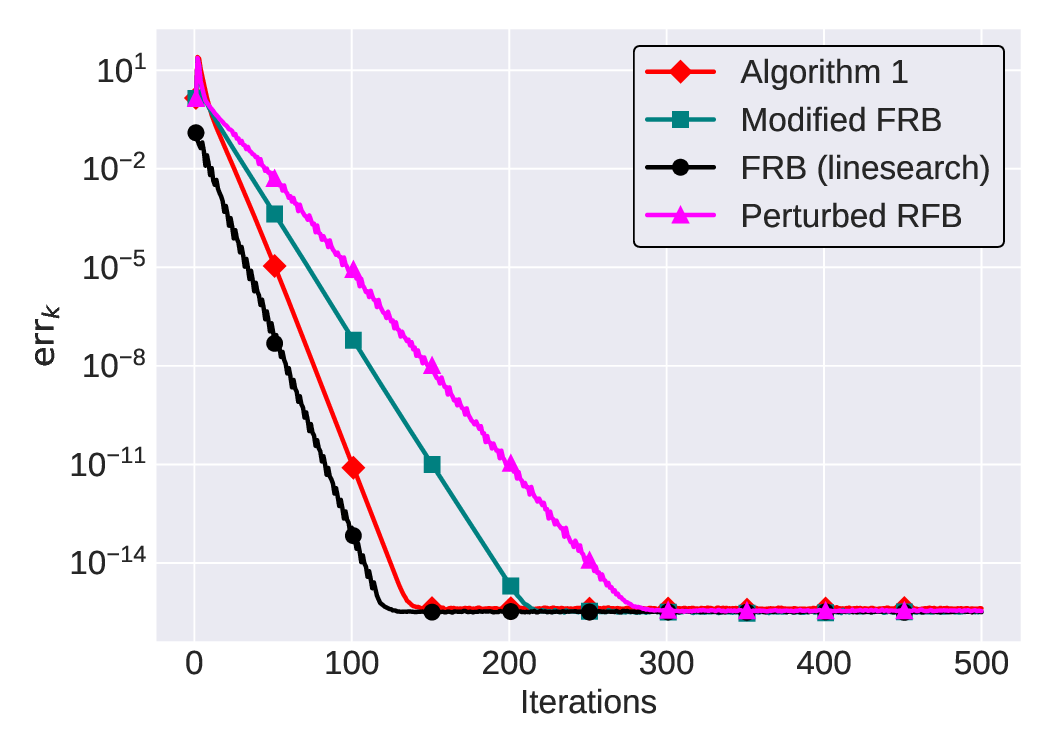}
    \caption{$m=1000$}
    \label{fig:ex2_delta_m500}
  \end{subfigure}\hfill
  \begin{subfigure}[b]{0.32\textwidth}
    \centering
    \includegraphics[width=\linewidth]{example_2_four_algorithms_m_2000_delta_.001_alpha_.01}
    \caption{$m=1000$}
    \label{fig:ex2_delta_m1000}
  \end{subfigure}

  \vspace{1em}

  \begin{subfigure}[b]{0.32\textwidth}
    \centering
    \includegraphics[width=\linewidth]{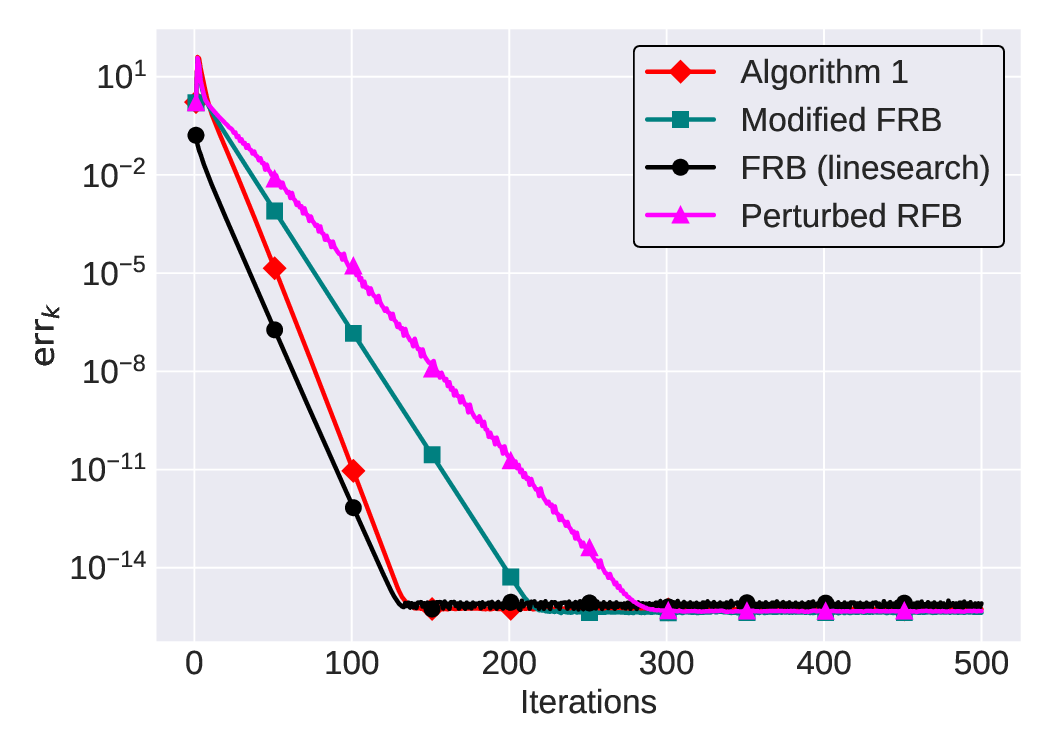}
    \caption{$m=2000$}
    \label{fig:ex2_algos_m300}
  \end{subfigure}\hfill
  \begin{subfigure}[b]{0.32\textwidth}
    \centering
    \includegraphics[width=\linewidth]{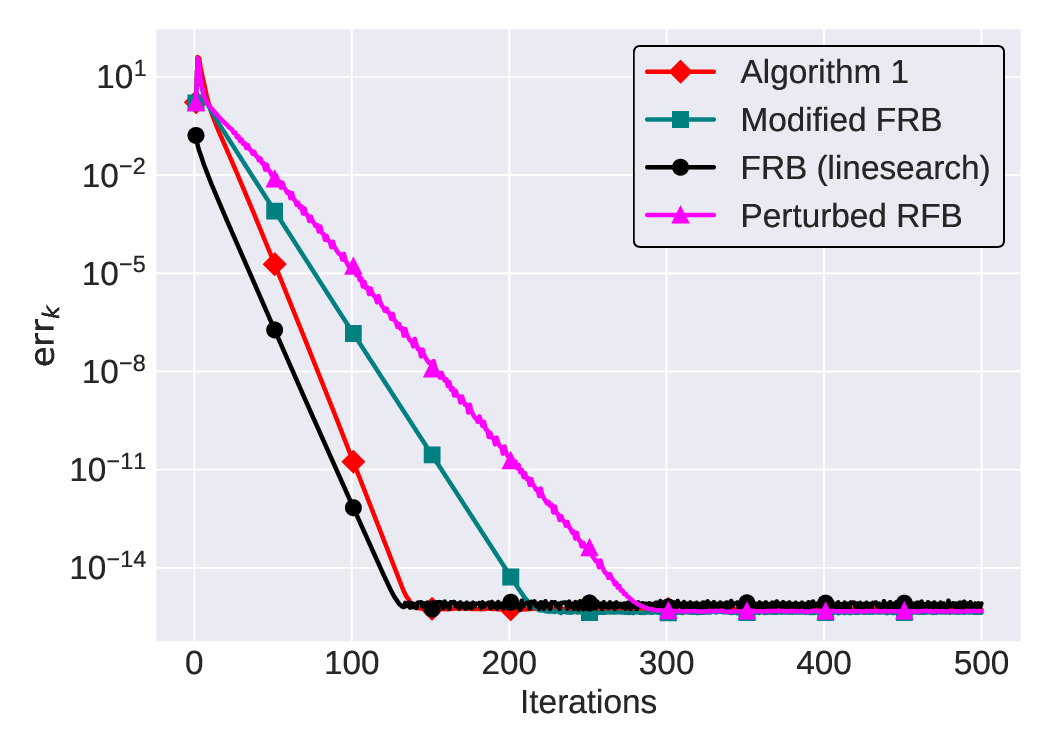}
    \caption{$m=2000$}
    \label{fig:ex2_algos_m500}
  \end{subfigure}\hfill
  \begin{subfigure}[b]{0.32\textwidth}
    \centering
    \includegraphics[width=\linewidth]{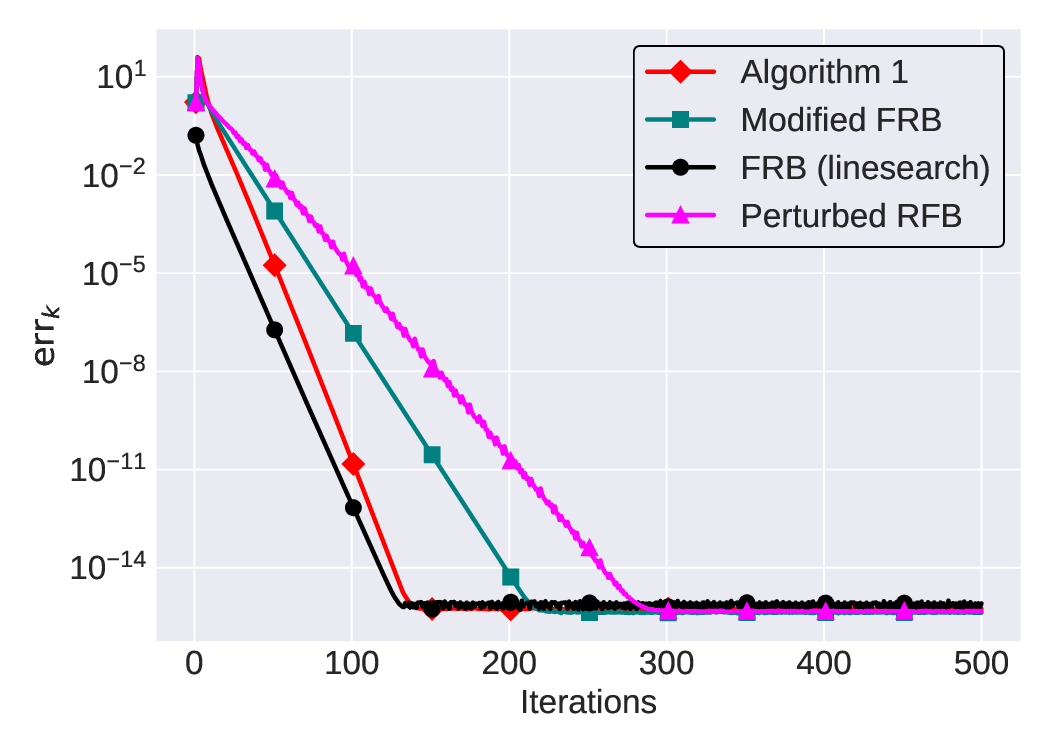}
    \caption{$m=2000$}
    \label{fig:ex2_algos_m1000}
  \end{subfigure}

\caption{Comparison of $\mathrm{err}_k$ for Example~\ref{example_2} for $m\in\{1000,5000\}$. 
The parameter choices in Algorithm~\ref{alg:forb_incr} are selected as: for (a) and (d), \((\alpha,\delta)=(0,0)\); for (b) and (e), \((\alpha,\delta)=(10^{-2},10^{-2})\); for (c) and (f) \((\alpha,\delta)=(10^{-2},10^{-3})\).}
  \label{fig:example2_delta_and_comparison}
\end{figure}
Before proving the monotonicity of $B$, notice that $v^\top S v = 0$ for all $v$ when $S$ is skew-symmetric (i.e., $S^\top=-S$). Thus, only the symmetric part of $M$ plays the role, which is 
$$
\mathrm{sym}(M) := \tfrac12(M+M^\top) = \tfrac12(G+G^\top)+D.
$$
Now, we choose $D=\tau I$, where $\tau \ge -\lambda_{\min}(\mathrm{sym}(G))+\varepsilon$ and $\varepsilon>\lambda_{\min}(\mathrm{sym}(G))$. Therefore, $\mathrm{sym}(M)\succeq 0$ and, for all $x,y$, we have
$$
\langle x-y,\,B(x)-B(y)\rangle
= \langle x-y,\,\mathrm{sym}(M)(x-y)\rangle\ge 0,
$$
hence $B$ is monotone. Furthermore, we have
$$
\|B(x)-B(y)\| = \|M(x-y)\| \le \|M\|_2\,\|x-y\|,
$$
so $B$ is $L$–Lipschitz with $L=\|M\|_2$.

Let $E=PNP^\top$ be an orthogonal decomposition with $N=\mathrm{diag}(\lambda_i)$. Then for any $\lambda>0$ and $z\in\mathbb{R}^m$, we have 
$$
J_{\lambda A}(z) \;=\; (I+\lambda(E+\beta I))^{-1}z
\;=\; P\;\mathrm{diag}\:\Big(\tfrac{1}{1+\lambda(\beta+\lambda_i)}\Big)\,P^\top z.
$$
We fix the \emph{seed} to be $10$, and draw the elements of
$R,G,\widetilde R\in\mathbb{R}^{m\times m}$ and $b\in\mathbb{R}^{m}$ from $\mathcal{N}(0,1)$. According to our requirements, we set $E=\tfrac12(R+R^\top)$ and 
$S=\tfrac12(\bar R- \bar R^\top)$. Note that $E=E^\top$ and $S^{\top}=-S$. We initialize all the algorithms with $x_0=x_1=0$ and other parameters are selected as follows.

\begin{itemize}
    \item \textbf{Algorithm \ref{alg:forb_incr}}: $\varepsilon = 10^{-12}$, $\lambda_0=0.3,\lambda_{-1} =0.1$, $c_2= \frac{0.9(1-\varepsilon-\alpha)}{2|\delta|+2}$ for given $\alpha,\delta$, and $c_1=0.9c_2$, $\gamma_k =\frac{0.1}{k^{1.001}}$.\
    \item \textbf{Modified FRB} \cite{hieu2021modified}:  $\lambda_0=0.3,\lambda_{-1} =0.1$ and $\mu=0.4$ as defined in the algorithm.
    \item \textbf{FRB (linesearch)} \cite{malitsky2020forward}: $\lambda_0=0.3,\lambda_{-1} =0.1$ and $\delta=0.9, \sigma=0.5$ used inside the loop of the linesearch.
    \item \textbf{Perturbed RFB} \cite{tan2025perturbed}:  $\lambda_0=0.3,\lambda_{-1} =0.1$, $\mu=0.19$ and $\gamma_k =\frac{0.1}{k^{1.001}}$.
\end{itemize}
All numerical results for this example are presented in Figure~\ref{fig:example2_delta_and_comparison}, with the corresponding iteration counts and CPU times reported in Tables~\ref{tab:ex2_benchmark_alpha0_delta0} and~\ref{tab:ex2_benchmark}. From Figure~\ref{fig:example2_delta_and_comparison}, we see that the FRB method with linesearch can be slightly faster in terms of iteration-wise decrease than Algorithm~\ref{alg:forb_incr}. However, Tables~\ref{tab:ex2_benchmark_alpha0_delta0} and~\ref{tab:ex2_benchmark} show that Algorithm~\ref{alg:forb_incr} achieves substantially lower CPU time than the linesearch variant and the other competing methods.

\end{example}

\begin{table}[htbp]
\centering
\caption{Iteration counts and CPU time (in seconds) to reach tolerance \(10^{-10}\) for Example~\ref{example_2} with different dimensions \(m\). For Algorithm~\ref{alg:forb_incr}, we use \(\alpha=\delta=0\).}
\label{tab:ex2_benchmark_alpha0_delta0}
\begin{tabular}{lcccccc}
\toprule
 & \multicolumn{2}{c}{\(m=200\)} & \multicolumn{2}{c}{\(m=500\)} & \multicolumn{2}{c}{\(m=700\)} \\
\cmidrule(lr){2-3}\cmidrule(lr){4-5}\cmidrule(lr){6-7}
Algorithm & Iter. & CPU (s) & Iter. & CPU (s) & Iter. & CPU (s) \\
\midrule
Alg.~1 (FoRB inc.)   & 84  & 0.019643  & 87  & 0.090670  & 90  & 0.203289 \\
Modified FRB         & 127 & 0.025333  & 131 & 0.082482  & 133 & 0.435675 \\
FRB (linesearch)     & 129 & 0.047840  & 92  & 0.091785  & 82  & 0.552501 \\
Perturbed RFB        & 165 & 0.048392  & 178 & 0.123114  & 180 & 0.719887 \\
\bottomrule
\end{tabular}
\end{table}

\begin{table}[htbp]
\centering
\caption{Iteration counts and CPU time (in seconds) to reach tolerance \(10^{-10}\) for Example~\ref{example_2} with different dimensions \(m\). For Algorithm~\ref{alg:forb_incr}, we use \(\alpha=\delta=10^{-2}\).}
\label{tab:ex2_benchmark}
\begin{tabular}{lcccccc}
\toprule
 & \multicolumn{2}{c}{\(m=1000\)} & \multicolumn{2}{c}{\(m=2000\)} & \multicolumn{2}{c}{\(m=3000\)} \\
\cmidrule(lr){2-3}\cmidrule(lr){4-5}\cmidrule(lr){6-7}
Algorithm & Iter. & CPU (s) & Iter. & CPU (s) & Iter. & CPU (s) \\
\midrule
Algorithm~\ref{alg:forb_incr}  & 93  & 0.265631  & 96  & 1.402477  & 96  & 4.032530 \\
Modified FRB         & 138 & 0.607342  & 144 & 3.093450  & 144 & 7.362841 \\
FRB (linesearch)     & 72  & 0.569167  & 119 & 4.574074  & 118 & 10.065574 \\
Perturbed RFB        & 183 & 0.844361  & 187 & 5.658181  & 184 & 12.295901 \\
\bottomrule
\end{tabular}
\end{table}

\subsection{LASSO signal recovery}
We consider the LASSO signal recovery problem
\begin{equation}\label{eq:lasso}
    \min_{x\in\mathbb{R}^{n}}~ \tfrac12\|A x -y\|_2^{\,2}+\lambda\|x\|_1,
\end{equation}
where $A\in\mathbb{R}^{m\times n}$, $y\in\mathbb{R}^{m}$ is the observation vector, and $\lambda>0$ is the regularization parameter. The $\ell_1$-term promotes sparsity of the recovered signal. Identifying \eqref{eq:lasso} with the monotone inclusion
model~\eqref{main}, we set $A=\partial(\lambda\|\cdot\|_1)$ and $B(x)=A^\top(Ax-y)$, where $\partial(\cdot)$ denotes the
subdifferential.
\begin{figure}[htbp]
  \centering
  \begin{subfigure}[b]{0.45\textwidth}
    \centering
    \includegraphics[width=\linewidth]{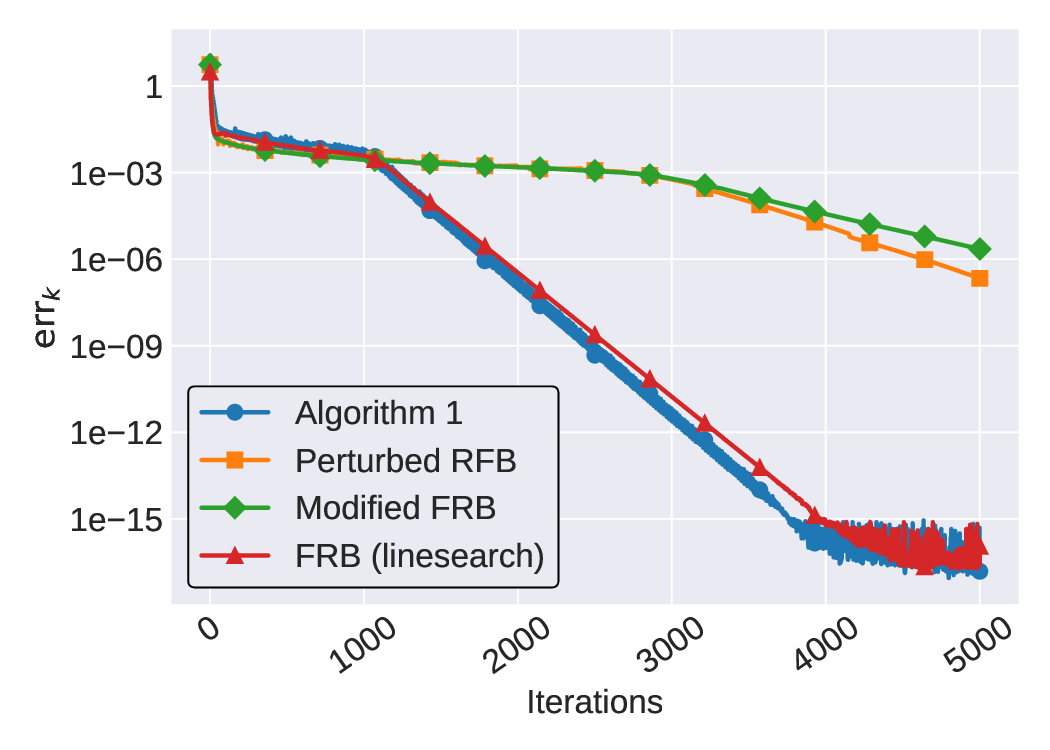}
    \caption{Iterate gap.}
  \end{subfigure}\hfill
  \begin{subfigure}[b]{0.45\textwidth}
    \centering
    \includegraphics[width=\linewidth]{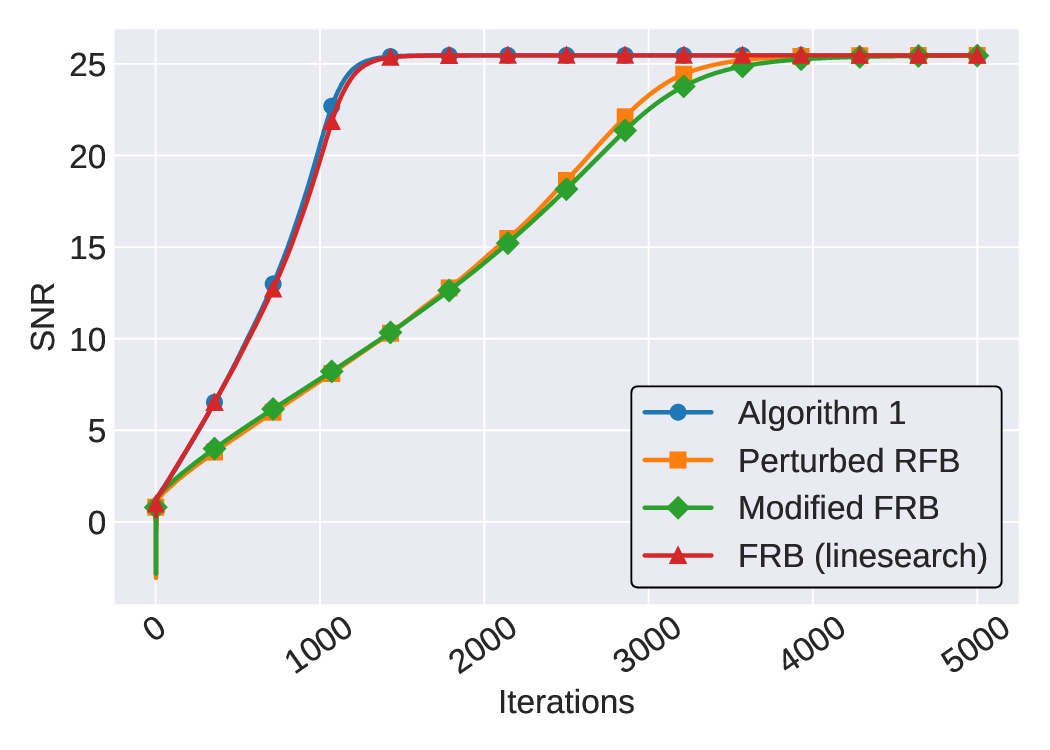}
    \caption{SNR versus iteration.}
  \end{subfigure}
  \caption{Convergence plots for LASSO signal recovery: (a) iterate gap versus iterations, and (b) SNR recovery over iterations.}
  \label{fig:conv_gap_snr}
\end{figure}
In all tests we fix $m=256$, $n=1024$, and the sparsity parameter $d=60$. The entries $(A_{ij})$ of the matrix $A$ are generated from $\mathcal{N}(0,1/m)$. The elements of the ground-truth sparse signal $x^*$ is first sampled from
$\mathcal{N}(0,1)$, after which we choose $d$ indices uniformly at random to be nonzero and set the remaining components to zero. The noisy measurements are produced according to
\[
y = Ax^* + \varepsilon,\qquad \varepsilon\sim\mathcal{N}(0,\sigma^2 I_m),\ \sigma=0.01.
\]
After a moderate grid search, we set $\lambda=0.01$. Our objective is to recover $x^*$ by applying all algorithms to this instance. To monitor reconstruction quality, at iteration $k$ we report the signal-to-noise ratio (SNR),
\[
\mathrm{SNR}_k \;=\; 20\log_{10}\!\left(\frac{\|x^*\|_2}{\|x_k-x^*\|_2}\right),
\]
where $x_k$ denotes the estimate produced at the $k$-th iteration. Figures~\ref{fig:true_vs_rec_all} and~\ref{fig:coef_rec_all}
show that all methods nearly recover the nonzero components of $x^*$. Moreover,
Figure~\ref{fig:alg1_coeffs_6panels} displays the reconstructions obtained by Algorithm~\ref{alg:forb_incr} for several choices of
$(\alpha,\delta)$; in each case the recovery remains accurate, indicating robustness with respect to these parameters. Finally, as
seen in Figure~\ref{fig:conv_gap_snr}, Algorithm~\ref{alg:forb_incr} reaches high-quality reconstructions more rapidly than the
other methods, both in terms of the iterate gap and the SNR evolution.
\begin{figure}[htbp]
  \centering
  \begin{subfigure}[b]{0.4\textwidth}
    \centering
    \includegraphics[width=\linewidth]{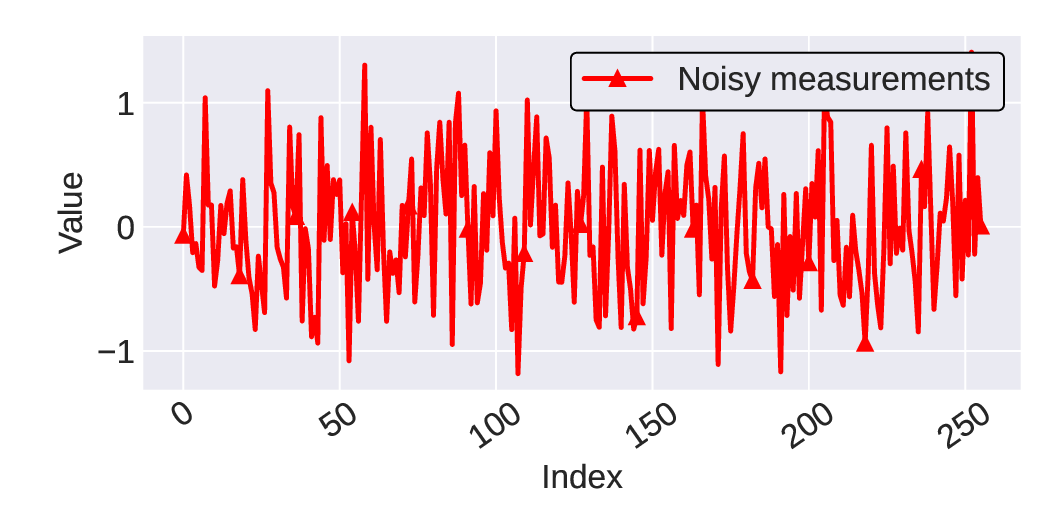}
    \caption{Noisy Signal.}
  \end{subfigure}\hfill
  \begin{subfigure}[b]{0.4\textwidth}
    \centering
    \includegraphics[width=\linewidth]{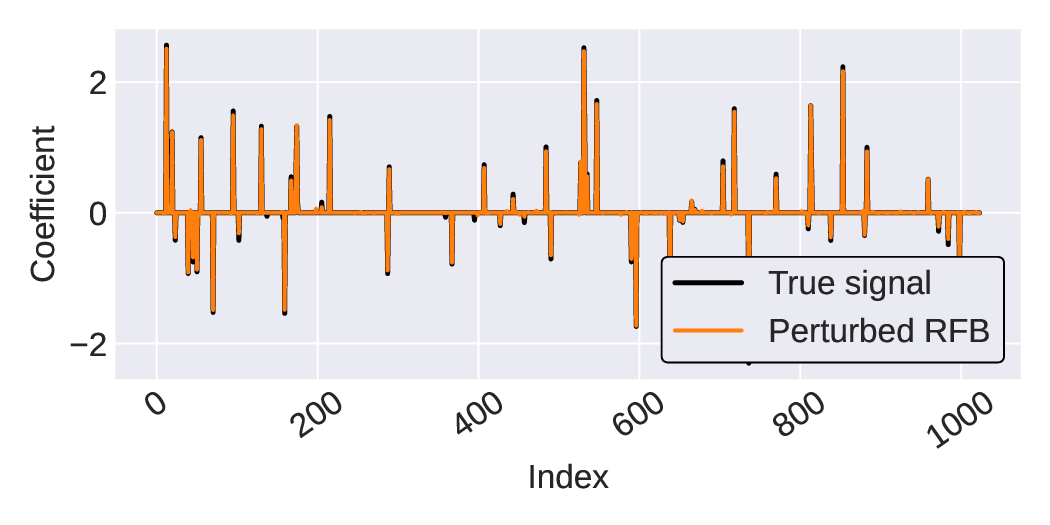}
    \caption{Perturbed RFB}
  \end{subfigure}

  \vspace{0.8em}

  \begin{subfigure}[b]{0.46\textwidth}
    \centering
    \includegraphics[width=\linewidth]{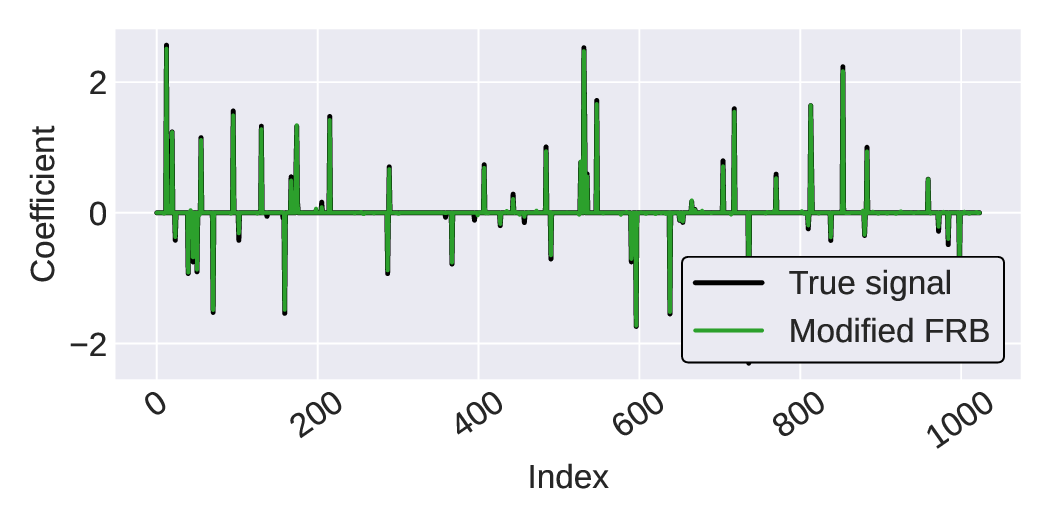}
    \caption{Modified FRB}
  \end{subfigure}\hfill
  \begin{subfigure}[b]{0.46\textwidth}
    \centering
    \includegraphics[width=\linewidth]{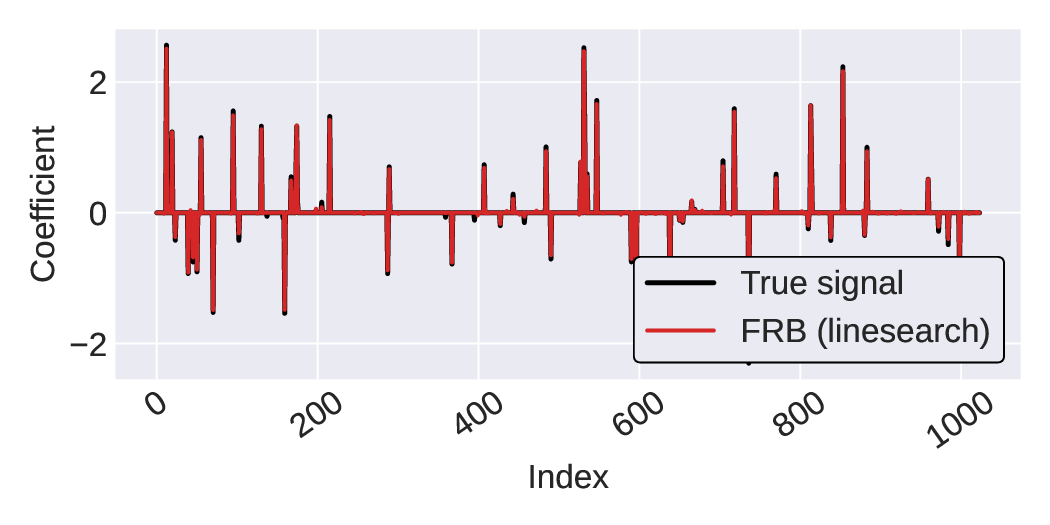}
    \caption{FRB (linesearch)}
  \end{subfigure}

  \caption{Comparison of the true sparse signal $x^*$ reconstructed by different algorithms.}
  \label{fig:true_vs_rec_all}
\end{figure}

\begin{figure}[htbp]
  \centering

  \begin{subfigure}[b]{0.4\textwidth}
    \centering
    \includegraphics[width=\linewidth]{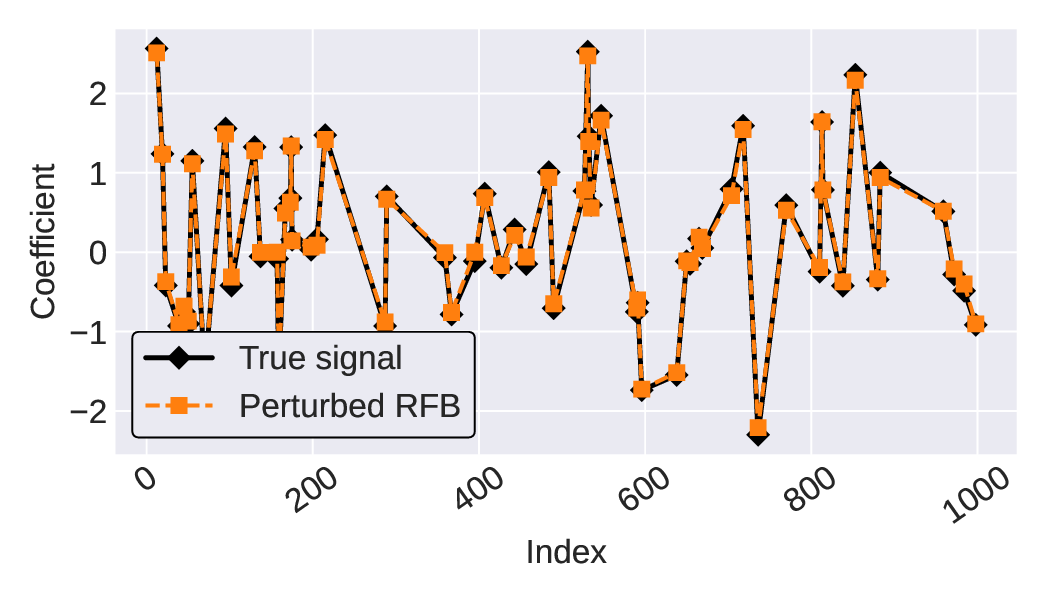}
    \caption{Perturbed RFB.}
  \end{subfigure}\hfill
  \begin{subfigure}[b]{0.4\textwidth}
    \centering
    \includegraphics[width=\linewidth]{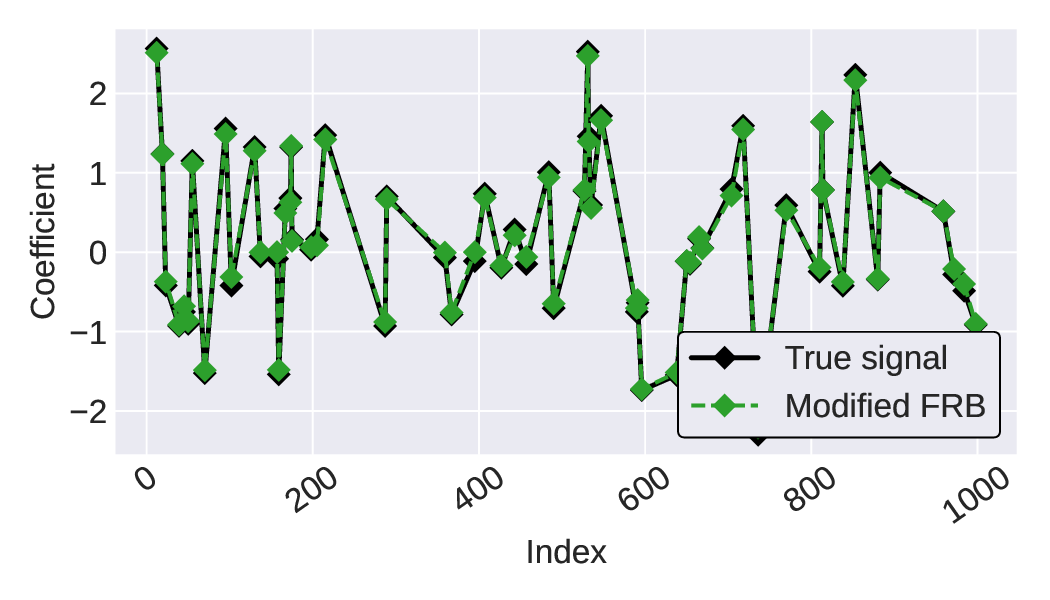}
    \caption{Modified FRB.}
  \end{subfigure}

  \vspace{0.8em}

  \begin{subfigure}[b]{0.4\textwidth}
    \centering
    \includegraphics[width=\linewidth]{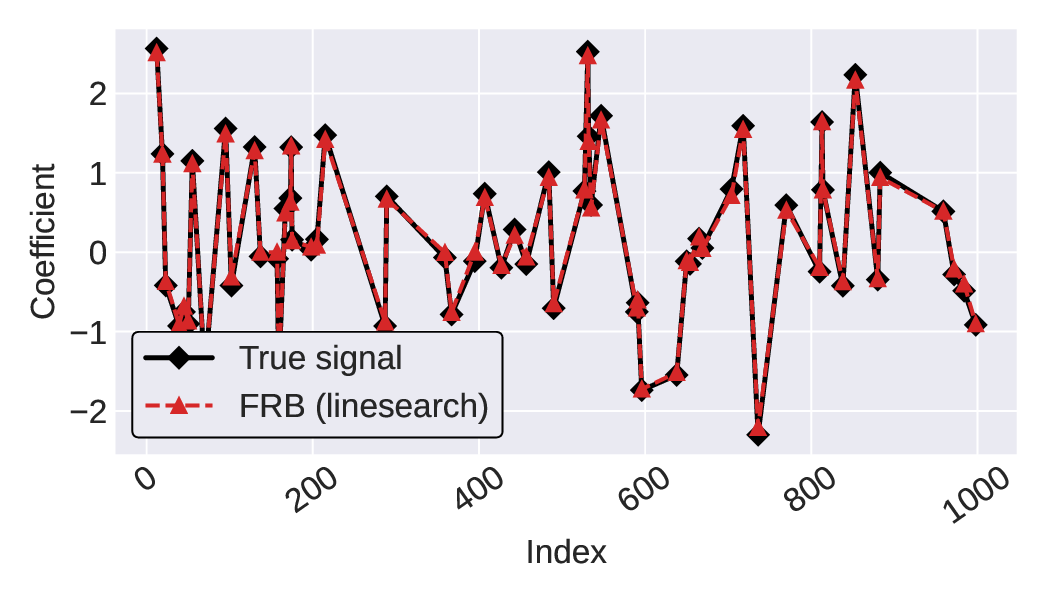}
    \caption{FRB (linesearch).}
  \end{subfigure}

  \caption{Coefficient recovery of the true signal $x^*$ by various algorithms.}
  \label{fig:coef_rec_all}
\end{figure}


\begin{figure}[htbp]
  \centering

  \begin{subfigure}[b]{0.32\textwidth}
    \centering
    \includegraphics[width=\linewidth]{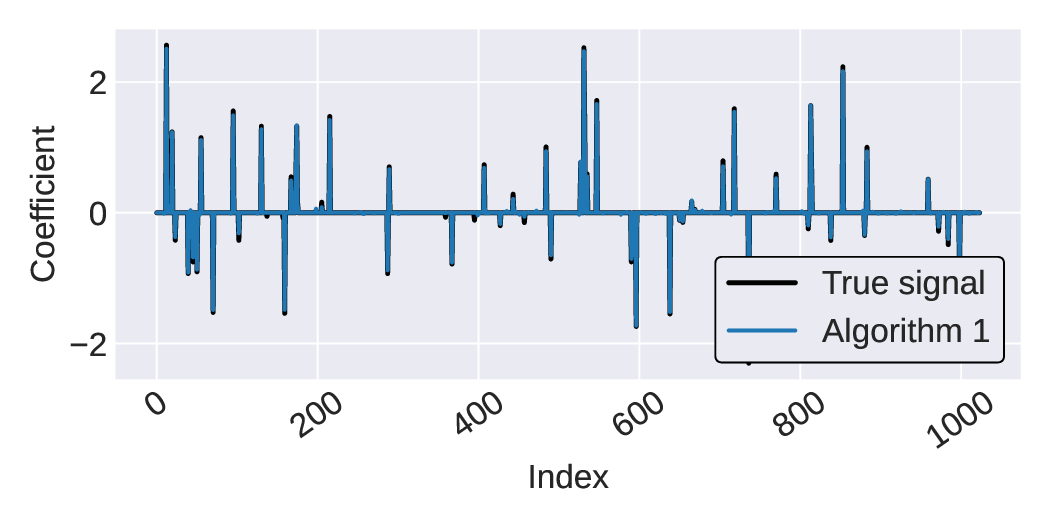}
    \caption{$(\alpha,\delta)=(0,0)$}
  \end{subfigure}\hfill
  \begin{subfigure}[b]{0.32\textwidth}
    \centering
    \includegraphics[width=\linewidth]{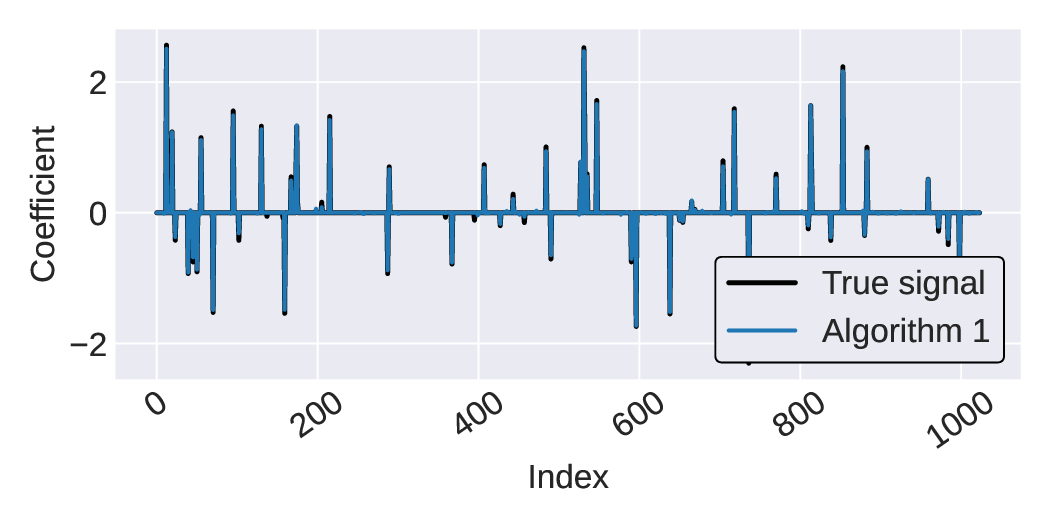}
    \caption{$(\alpha,\delta)=(10^{-2},10^{-3})$}
  \end{subfigure}\hfill
  \begin{subfigure}[b]{0.32\textwidth}
    \centering
    \includegraphics[width=\linewidth]{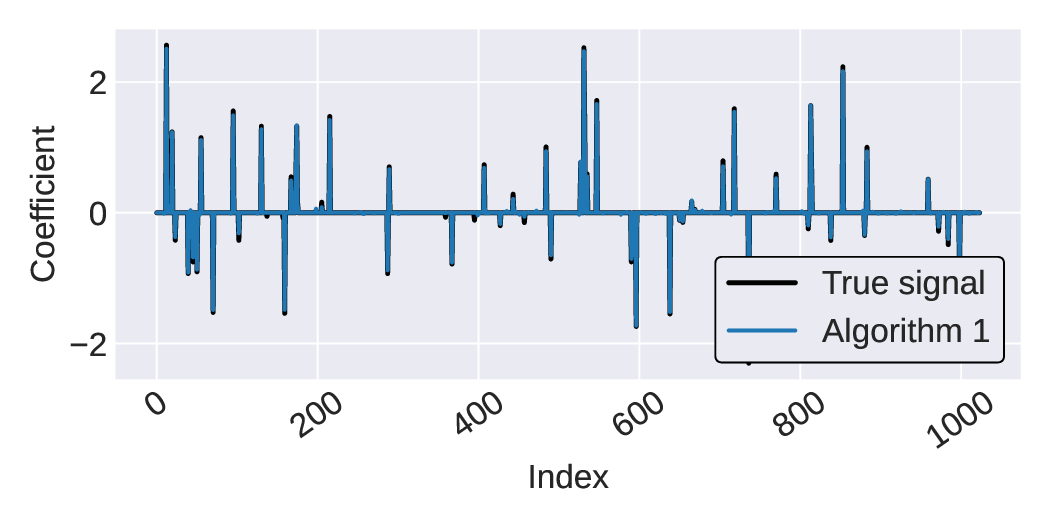}
    \caption{$(\alpha,\delta)=(10^{-2},10^{-2})$}
  \end{subfigure}

  \vspace{0.8em}

  \begin{subfigure}[b]{0.32\textwidth}
    \centering
    \includegraphics[width=\linewidth]{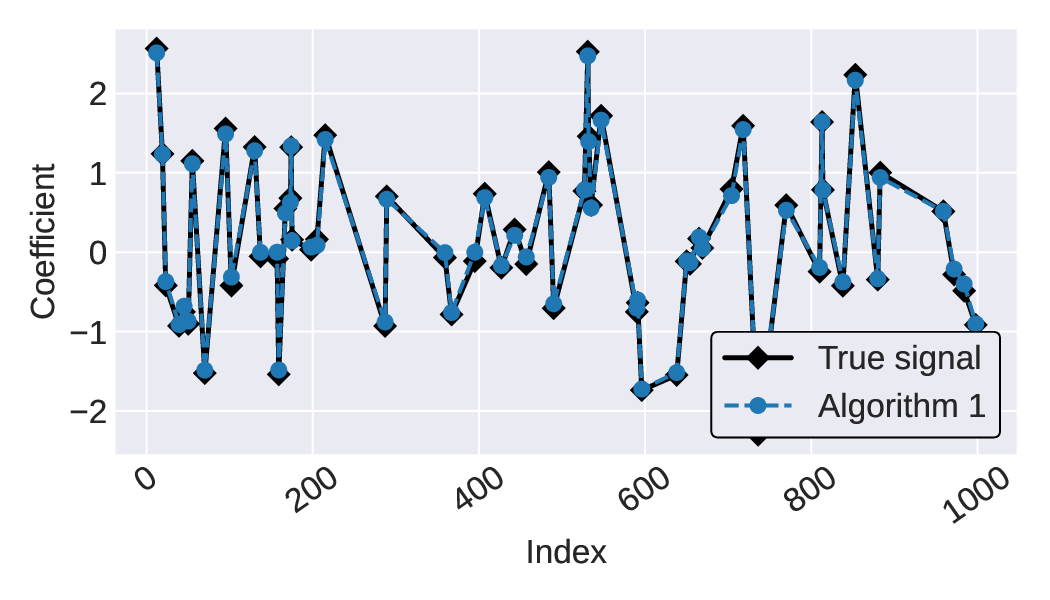}
    \caption{Algorithm \ref{alg:forb_incr}}
  \end{subfigure}\hfill
  \begin{subfigure}[b]{0.32\textwidth}
    \centering
    \includegraphics[width=\linewidth]{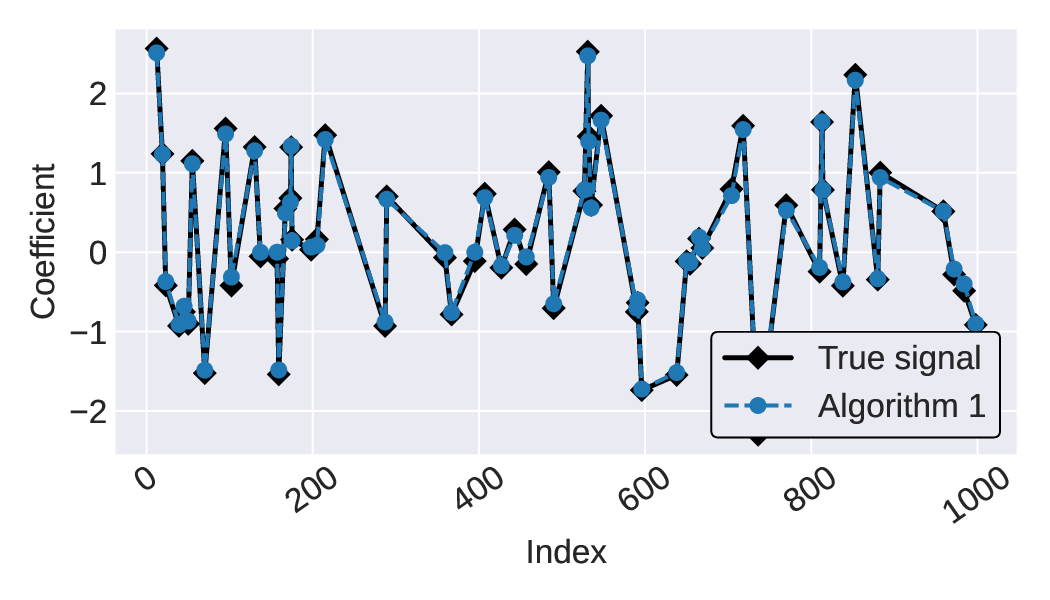}
    \caption{Algorithm \ref{alg:forb_incr}}
  \end{subfigure}\hfill
  \begin{subfigure}[b]{0.32\textwidth}
    \centering
    \includegraphics[width=\linewidth]{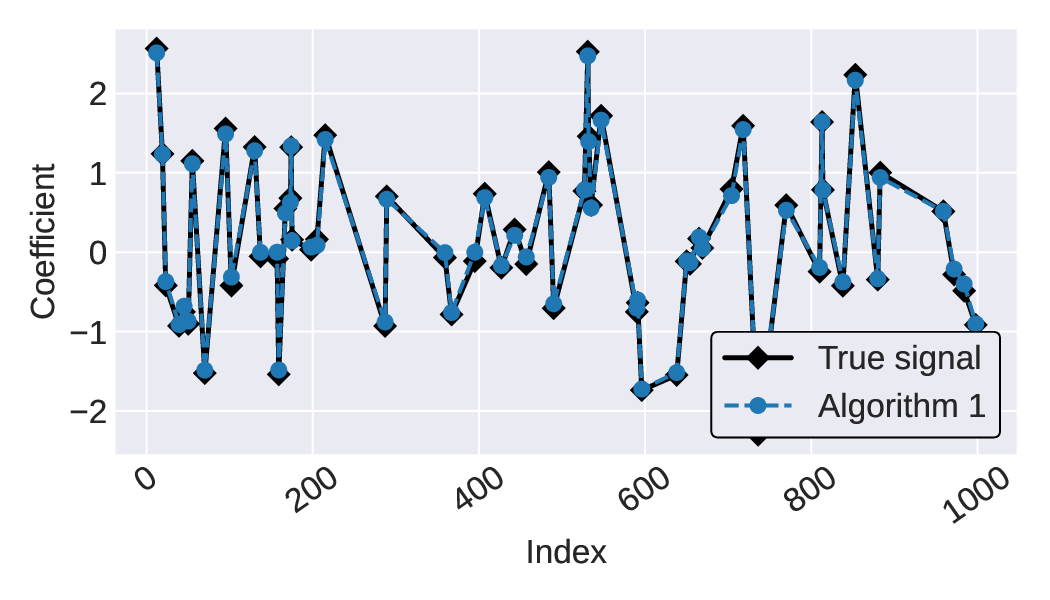}
    \caption{Algorithm \ref{alg:forb_incr}}
  \end{subfigure}

  \caption{Effect of $(\alpha,\delta)$ in Algorithm \ref{alg:forb_incr} on the LASSO signal recovery (top row) problem and the corresponding coefficient recovery on the true signal (bottom row).}
  \label{fig:alg1_coeffs_6panels}
\end{figure}

\section{Conclusion}
We studied the monotone inclusion \eqref{main} via the generalized forward–reflected–backward (GFRB) method, an extension of FRB, under a non-decreasing step–size rule, and established its convergence properties. On a particular example, we showed that the theoretical linear rate for GFRB  is bounded below by $1/\sqrt{2}$, and on another example, we further demonstrated that a suitable choice of the initial parameters yields a strictly improved linear rate. In addition, we proposed an extended primal–dual twice–reflected algorithm (EPDTR), which generalizes PDTR to a class of structured inclusion problems that includes \eqref{main}. Our current analysis of EPDTR requires prior knowledge of the Lipschitz constant and the norm of the underlying linear operator. A natural direction for future work is to remove this dependence via linesearch or other adaptive methods, for instance, along the lines of~\cite{soe2025golden}, which by construction, would yield a parameter-free variant of PDTR.

\section*{Acknowledgements}
Santanu Soe gratefully acknowledges Prof. Matthew K. Tam (The University of Melbourne) for his guidance and unwavering support. The research of Santanu Soe was supported by the Prime Minister’s Research Fellowship program (Project number SB23242132MAPMRF005015), Ministry of Education, Government of India, and the Melbourne Research Scholarship.
\section*{Data Availability}
No datasets were generated or analyzed during the course of this study. The Python codes used in this paper are available from the corresponding author upon reasonable request.

\section*{Conflict of Interest}
The authors declare that there are no conflicts of interest in this paper.

\bibliography{My_bib}
\bibliographystyle{abbrv}

\end{document}